\documentclass[11pt]{amsart}

\usepackage{amsmath,amssymb,latexsym,soul,cite,mathrsfs,amsfonts}
\usepackage{color,enumitem,graphicx}
\usepackage[colorlinks=true,urlcolor=blue,
citecolor=red,linkcolor=blue,linktocpage,pdfpagelabels,
bookmarksnumbered,bookmarksopen]{hyperref}
\usepackage[english]{babel}
\usepackage{comment}
\usepackage[left=2.9cm,right=2.9cm,top=2.8cm,bottom=2.8cm]{geometry}
\usepackage[hyperpageref]{backref}

\usepackage[colorinlistoftodos]{todonotes}
\makeatletter
\providecommand\@dotsep{5}
\def\listtodoname{List of Todos}
\def\listoftodos{\@starttoc{tdo}\listtodoname}
\makeatother

\usepackage{verbatim} % to using the command "\comment"

\numberwithin{equation}{section}
\newtheorem{theorem}{Theorem}[section]
\newtheorem{proposition}[theorem]{Proposition}
\newtheorem{lemma}[theorem]{Lemma}
\newtheorem{corollary}{Corollary}
\newtheorem{remark}{Remark}

\newcommand\restr[2]{{% we make the whole thing an ordinary symbol
  \left.\kern-\nulldelimiterspace % automatically resize the bar with \right
  #1 % the function
  \vphantom{\big|} % pretend it's a little taller at normal size
  \right|_{#2} % this is the delimiter
  }}
\title[On constrained SP system]
{%A fibering  approach to a Schr\"odinger-Poisson system with prescribed $L^2$ norm \\ OR \\
On the structure of the Nehari set associated to a
Schr\"odinger-Poisson system with prescribed mass: \\
old and new results}

\author[G. Siciliano]{Gaetano Siciliano}

\author[K. Silva]{ Kaye Silva}

\address[G. Siciliano]{\newline\indent
	Departamento de Matem\'atica - Instituto de Matem\'atica e Estat\'istica
	\newline\indent 
	Universidade de S\~ao Paulo
	\newline\indent
	Rua do Mat\~ao 1010,  05508-090  S\~ao Paulo, SP,  Brazil}
\email{\href{mailto:sicilian@ime.usp.br}{sicilian@ime.usp.br}}

\address[K. Silva]{\newline\indent
	Instituto de Matem\'atica e Estat\'istica.   
	\newline\indent 
	Universidade Federal de Goi\'as,
	\newline\indent
Rua Samambaia, 74001-970, Goi\^ania, GO, Brazil}
\email{\href{mailto:kayesilva@ufg.br}{kayesilva@ufg.br}}

\thanks{Gaetano Siciliano  was partially supported by
Fapesp grant 2018/17264-4, CNPq grant 304660/2018-3 and Capes. 
Kaye Silva was partially supported by CNPq/Brazil under Grant 408604/2018-2.}


\begin{document}

\begin{abstract}
In this paper we apply the fibering method of Pohozaev
to a Schr\"odinger-Poisson system, with prescribed $L^{2}$
norm of the unknown, in the whole $\mathbb R^{3}$.
The method makes clear the role played by the exponents $p=3, p=8/3, p=10/3$.

Beside to show as old results can be obtained in a unified way,
we exhibit also new ones.

\end{abstract}

\bigskip

\maketitle
\begin{center}
\begin{minipage}{12cm}
\tableofcontents
\end{minipage}
\end{center}

\bigskip

\maketitle

\section{Introduction}

It is well-known that the following  Schr\"odinger equation
(where all the physical constants are normalised to unity)
\begin{equation}\label{eq:S}
i\partial_{t} \psi = -\Delta_{x} \psi  + q \left(|\cdot |^{-1} * |\psi|^{2} \right) \psi -\lambda |\psi |^{8/3}\psi,
\quad \psi: \mathbb R^{3}\times\mathbb R\to \mathbb C,
\end{equation}
has a relevant role in many physical models.
Here $i$ is the imaginary unit, $\Delta_{x}$ is the Laplacian with respect to the spatial
variables,  $*$ is the $x-$convolution and $q,\lambda$ are positive parameters.

As we can see, two types of potentials, different in nature, appear in the equation:
the first one is the Hartree (or Coulomb) potential  given by $V_{\textrm{H}}(\cdot, t) = |\cdot |^{-1} * |\psi(\cdot, t)|^{2}$
which is nonlocal and the second one is 
 the Slater approximation of the exchange term, 
 given by $|\psi |^{8/3}\psi$, which is local, although nonlinear.
 In this context $q$ and $\lambda$ are also called,
 respectively, {\sl Poisson constant} and {\sl Slater constant}.
The nonlocal potential can be seen as ``generated'' by the same wave function $\psi$,
in virtue of the Poisson equation
$$-\Delta_{x} V_{\textrm{H}} = 4\pi |\psi|^{2}.$$
A particular feature of \eqref{eq:S} is that, due to the invariance by $U(1)$ gauge-transformations
and the invariance by  time translations,  by the Noether theorem, on the solutions $\psi$ the quantities 
$$M(\psi) (t) = \int | \psi(x,t)|^{2}dx$$
and 
$$E(\psi)(t)= \frac{1}{2} \int |\nabla_{x} \psi(x,t)|^{2}dx+\frac{ q}{4}\int \left( \frac{1}{|\cdot |} * |\psi(\cdot,t)|^{2} \right) |\psi(x,t)|^{2} dx
 - \frac{3\lambda}{8}\int |\psi(x,t)|^{8/3}dx$$
are conserved in time. In physical terms they are called respectively {\sl mass} and  {\sl energy}
of the solution.

Since the Hartree potential and the Slater term have different 
sign in the energy functional, they are in competition and then a different
behaviour of $E$ is expected depending on the values of  the parameters $q$ and $\lambda$.
 For more physical details and the derivation of \eqref{eq:S}
 see e.g. the seminal works
  \cite{bokanoetal,liebsimon, mauser, PY, Slater} and the references therein.
We mention  that the above equation have been derived also
in the framework of Abelian Gauge Theories in \cite{bf} and called {\sl Schr\"odinger-Poisson
system}.

\medskip

In this work we consider the problem of finding  standing waves solutions $\psi(x,t) = u(x) e^{-i\ell t}$
$$u:\mathbb R^{3} \to \mathbb R, \quad \ell \in \mathbb R$$
to the above equation \eqref{eq:S}  under the
mass  constraint (as justified by the mass conservation law)
and where the exponent $8/3$ is replaced by a more general $p$.
More specifically
we consider the problem of finding 
$\ell \in \mathbb R$ and $u\in H^{1}(\mathbb R^{3})$
satisfying
\begin{equation}\label{eq:problem}
\left \{ \begin{array}{ll}
-\Delta u + q\phi_{u} u -\lambda |u|^{p-2}u =\ell u, & \mbox{ in  } \mathbb R^{3}  \medskip \\
\displaystyle\int u^{2} = r
\end{array}\right.
\end{equation}
(from now on all the integrals will be in $\mathbb R^{3}$ and $dx$ will be omitted)
where 
\begin{itemize}
\item $q,\lambda, r>0$ are given parameters,
\item $p\in(2,6)$,
\item $\phi_{u} \in D^{1,2}(\mathbb R^{3})$ is the unique solution of the Poisson equation $-\Delta \phi=4\pi u^{2}$
in $\mathbb R^{3}$, that can be  represented, for $u\in H^{1}(\mathbb R^{3})$, as
$$\phi_{u} = \frac{1}{|\cdot| } * u^{2}.$$
\end{itemize}
 In particular we are interested in finding {\sl ground state solutions} $u$,
 namely the solutions with minimal energy in the sense specified below.

The usual way to attack the problem is by variational methods.
Indeed the weak solutions of equation \eqref{eq:problem} are easily seen to be critical points
of the energy functional
\begin{equation*}
E(u):=E_{q,\lambda}(u) = \frac{1}{2}\int|\nabla u|^2+\frac{q}{4}\int\phi_uu^2-\frac{\lambda}{p}\int|u|^p, 
\end{equation*}
constrained to the $L^2(\mathbb{R}^3)$ sphere
\begin{equation*}
S_r=\left\{u\in H^1(\mathbb{R}^3):\int u^2=r\right\},
\end{equation*}
as it follows by the Lagrange multiplier rule; in this case $\ell\in \mathbb R$
is the Lagrange multiplier associated to the critical point.
%
%
%where $p\in (2,6)$, $r,q,\lambda>0$ and for each $u\in H^1(\mathbb{R}^3)$ the function $\phi_u\in D^{1,2}(\mathbb{R}^3)$ is the unique weak solution of  $-\Delta\phi_u=4\pi u^2$. For each $q,\lambda>0$, define $E:=E_{q,\lambda}:H^1(\mathbb{R}^3)\to \mathbb{R}$ by
%\begin{equation*}
%E(u)=\frac{1}{2}\int|\nabla u|^2+\frac{q}{4}\int\phi_uu^2-\frac{\lambda}{p}\int|u|^p,
%\end{equation*}
%and for each $r>0$ denote
%\begin{equation*}
%S_r=\left\{u\in H^1(\mathbb{R}^3):\int |u|^2=r\right\}.
%\end{equation*}
%
%
%If $u\in H^1(\mathbb{R}^3)$ is a critical point to $E$ constrained to $S_r$, then from the Lagrange's multiplier rule, there exists $\ell_r\in \mathbb{R}$ such that
%\begin{equation}\label{SE}
%-\Delta u+q\phi_uu-\lambda|u|^{p-2}u+\ell_ru=0,
%\end{equation}
%which is a type of Schr\"odinger-Maxwell equations. Equation \eqref{SE} is related to the study of one-dimensional reduction of electron density in plasma physics, as well as in semiconductor theory, as a correction to the  Schr\"odinger-Poisson system, which corresponds to the case $q=0$. For more information with respect to the physical part see Catto et al. \cite{cattoetal}, Bokanowski et al. \cite{bokanoetal}, Lieb and Simon \cite{liebsimon} and Mauser \cite{mauser}. 
%
%Physically speaking, $E$ represents the energy and   the $L^2$ norm represents the mass, 
Then this problem  fit into the  question of finding critical points of the energy on the mass constraint
(see \cite{Bmilan}).

An interesting problem is the search of {\sl ground states solutions}, namely
the  minima of $E$ on $S_{r}$, since they give rise
to stable standing waves solutions for the evolution problem 
\eqref{eq:S}. The problem is not trivial since the behaviour of $E$
depends on $q,\lambda, p$ but actually also
the value $r$ has a main role. 

The search of minima has been addressed by Lions in the celebrated papers
\cite{lions1} %,lions2x}
 where he studied the problem from a mathematical point of view
and established, roughly speaking, that the existence of minimisers is 
equivalent to the strict sub-additive inequality 
\begin{equation}\label{SAI}
\inf_{S_{r}} E < \inf_{S_{s}} E + \inf_{S_{r-s}} E,  %E_r<E_s+E_{r-s}, 
\ \ 0<s<r.
\end{equation} 
%where
%\begin{equation*}
%E_\alpha:=\inf\{E(u):u\in S_\alpha\}.
%\end{equation*}
%the problem consists in find $u\in S_r$ such that $E_r=E(u)$. 
%As was already noted by Lions  \cite{lions1,lions2x}, the existence of minimizers to $E_r$, when $E_r<0$, is in general, equivalent to the strict sub-additive inequality 
%\begin{equation}\label{SAI}
%E_r<E_s+E_{r-s}, 
%\end{equation} 
%which by its 
In turn, this is equivalent to show that {\sl dichotomy} does not occur when one tries to apply the concentration-compactness principle of Lions (since the  vanishing is avoided due to
$\inf_{S_{r}} E<0$%$E_r<0$
).
As showed in recent papers,  inequality \eqref{SAI} does hold in certain intervals that depends on the values of $p$.
See e.g. Bellazzini and Siciliano \cite{bellasici1,bellasici2}, Catto and Lions \cite{cattolions}, Sanchez and Soler \cite{sanchezsoler}, Jeanjean and Luo \cite{jeanluo}, Catto et al. \cite{cattoetal} and Colin and Watanabe \cite{colinwatanabe}.

We point out that in the last decades   equations
like \eqref{eq:problem}, even without the mass constraint,
 have been extensively studied due  to the mathematical
 challenges raised by the nonlocal term $\phi_{u}$
 and its competition with the local nonlinearity.

\medskip

The aim of this paper is to establish,
by using the fibration method of Pohozaev developed in \cite{Poho} and the notion of extremal values introduced in Il'yasov \cite{Ily},
 a general framework which permits us to search for global minimizers of $E$ over components of 
 a suitable {\sl  Nehari type set}.
These components are shown to be differentiable manifolds and natural constraints for the energy functional. The method proposed in this work makes more clear the relation between minimizers of $E$ restricted to $S_r$ and the parameters $q,\lambda,p$ and $r$. 
In particular it gives some light on the role of special exponents $p$ appearing in the 
Schr\"odinger-Poisson system: $p=8/3, p=3, p=10/3$. Moreover it relates the strict sub-additive inequality with topological properties of some natural curves that crosses the Nehari manifolds as $r$ varies. 

Indeed beside recovering known results, we get also new ones and interesting estimates.

We think that this fibering approach can be used to solve also other different problems
involving suitable constraint (different from the $L^{2}-$norm).

\medskip

To conclude this Section, we  point out that recently the fibration method together with the notion of extremal values, that guarantees regions of parameters in which the Nehari manifold method can be applied to prove existence of solutions,
has been used successfully also in other 
types of equations as in 
\cite{kaye, SS, YK}.

\section{Statements of the  results}
In this paper we obtain four types of results, all  based on the introduction
of a suitable Nehari set type for the functional $E$ restricted to $S_{r}$
and on its properties.
Before to state our results, we need to introduce some notations.

First of all, by using standard methods, it is easy to see 
(see Proposition \ref{neh}) that every critical point of $E$ restricted to $S_r$ belongs to the {\sl Nehari type set}
\begin{equation*}
\mathcal{N}_r:=\mathcal{N}_{r,q,\lambda}=\left\{u\in S_r:\int |\nabla u|^2+\frac{q}{4}\int \phi_uu^2-\frac{3(p-2)}{2p}\lambda\int |u|^p=0\right\}
\end{equation*} 
Indeed this set has been already introduced in  \cite{JBL,jeanluo}. 
However,  by means of the fibering method,
we are able to  decompose $\mathcal N_{r}$ into the subsets
 \begin{eqnarray*}
 \mathcal{N}_r^+&=&\left\{u\in \mathcal{N}_r:\int |\nabla u|^2-\frac{3(p-2)(3p-8)}{4p}\lambda\int |u|^p>0\right\}, \\
 \mathcal{N}_r^0&=&\left\{u\in \mathcal{N}_r:\int |\nabla u|^2-\frac{3(p-2)(3p-8)}{4p}\lambda\int |u|^p=0\right\},\\
\mathcal{N}_r^-&=&\left\{u\in \mathcal{N}_r:\int |\nabla u|^2-\frac{3(p-2)(3p-8)}{4p}\lambda\int |u|^p<0\right\}
 \end{eqnarray*}
so that $\mathcal{N}_r=\mathcal{N}_r^+\cup \mathcal{N}_r^0\cup \mathcal{N}_r^-$  and 
even more,  we show that whenever nonempty, $\mathcal{N}_r^+$ and  $\mathcal{N}_r^-$ are differentiable
manifolds of codimension  $2$ in $H^{1}(\mathbb R^{3})$ and 
natural constraint for the functional $E$ restricted to $S_{r}$ (see Theorem \ref{naturalconstraints}).

\medskip 
 
%so that we have
%\begin{theorem}\label{naturalconstraints} Assume that $\mathcal{N}_r^+$ ($\mathcal{N}_r^-$) is non-empty. Then $\mathcal{N}_r^+$ $(\mathcal{N}_r^-)$ is a co-dimension $2$, $C^1$ manifold in $H^1(\mathbb{R}^3)$. Moreover $\mathcal{N}_r^+$ ($\mathcal{N}_r^-$) is a natural constraint for the functional $E$.
%\end{theorem}

One of the main ingredients in our proofs will be the analysis of the functional
\begin{equation*}
R_p(u)=\frac{\left(\displaystyle\int|\nabla u|^2\right)^{\frac{3p-8}{4(p-3)}}\left(q\displaystyle\int\phi_uu^2\right)^{\frac{10-3p}{4(p-3)}}}{\left(\lambda\displaystyle\int|u|^p\right)^{\frac{1}{2(p-3)}}},
\quad u\in S_{1}, \quad p\neq3.
\end{equation*}
This functional is obtained with the help of Pohozaev's fibration method and is the so-called nonlinear Rayleigh 
quotient introduced by  Il'yasov in \cite{Ily}. Its topological properties are related with existence and non-existence of 
solutions  for our problem and, although not in this form, this functional was already used in \cite{cattoetal}. See also \cite{colinwatanabe} where the nonlinear Rayleigh 
quotient was found by fixing $r>0$ and considering $q$ as a parameter; however
this is different from our approach, whose main goal is to analyse the 
 topological properties of the Nehari set also when $r$ varies. 

\bigskip

(I) Our first results concern with  new inequalities
we were not able to find in the literature.  They are obtained by exploring the functional $R_{p}$.
\begin{theorem}\label{th:desigualdadenova}
For each $p\in [10/3, 6)$ there exists a  constant $C_{q,\lambda,p}>0$ such that
$$\lambda\int |u|^{p}\leq C_{q,\lambda,p}\frac{\displaystyle \left(\int |\nabla u|^{2} \right)^{\frac{3p-8}{2}}\left(\displaystyle\int u^{2} \right)^{2(p-3)} }
{ \left( q\displaystyle\int\phi_{u}u^{2}\right)^{\frac{3p-10}{2}}} 
\, ,\quad \forall u\in H^1(\mathbb{R}^3)\setminus\{0\}.$$
\end{theorem}
We remark that  similar inequalities are known in the 
literature, see Catto et al. \cite{cattoetal} for the case $p\in[8/3,10/3]$.

We have also the following inequality whose proof will be more involved than the previous theorem.
\begin{theorem}\label{cattoinep<3} 
For each $p\in (2,3)$ there exists a  constant $C_{q,\lambda,p}>0$ such that
	\begin{equation*}
	\lambda \int |u|^p\ge C_{q,\lambda,p}\left(\int u^2\right)^{2(p-3)}\left(q\int \phi_uu^2\right)^{\frac{10-3p}{2}}\left(\int |\nabla u|^2\right)^{\frac{3p-8}{2}},\quad \forall u\in H^1(\mathbb{R}^3)\setminus\{0\}.
	\end{equation*}
\end{theorem}

\bigskip

(II) The second type of results deal with the structure of $\mathcal N_{r}$ and its consequences.
The situation will be different in the cases $p\neq3$ and $p=3$
and related existence/non-existence results are obtained.

\noindent {\bf The case $p\neq3$.}
For each  $p\in (2,6)\setminus\{ 3\}$, define the infimum of $E$ over a subset  of the Nehari set, by
\begin{equation*}
I_r:=I_{r,q,\lambda}=\inf\{E(u):u\in \mathcal{N}_r^+\cup \mathcal{N}_r^0\}.
\end{equation*}
In particular $\inf_{S_{r}}E\leq I_{r}$.
With our approach we are able to show the following.

\begin{theorem}\label{Nehari} There holds:
	\begin{enumerate}
		\item[i)] If $p\in (2,8/3)$, then $\mathcal{N}_r=\mathcal{N}_r^+\neq\emptyset$ and 
		$-\infty<I_r=\inf_{S_{r}}E<0$ for all $r>0$. \smallskip
		
		\item[ii)] If $p\in (10/3,6)$, then $\mathcal{N}_r=\mathcal{N}_r^-\neq\emptyset$ and $I_r=-\infty$ for all $r>0$. \smallskip
		\item[iii)] If $p=8/3$, then $\mathcal{N}_r=\mathcal{N}_r^+\neq\emptyset$ and 
		$-\infty<I_r=\inf_{S_{r}}E<0$ for all $r>0$. \smallskip
		\item[iv)] If $p=10/3$, then there exists a  constant $K_{\emph{GN}}>0$ such that $\mathcal{N}_r\neq\emptyset$ if, and only if, 
		\begin{equation*}
		\frac{5}{3K_{\emph{GN}}}\frac{1}{\lambda}<r^{{2}/{3}}.
		\end{equation*}
		In case $\mathcal{N}_r\neq \emptyset$ we have that $\mathcal{N}_r=\mathcal{N}_r^-$ and $I_r=-\infty$.
		\smallskip
		\item[v)] If $p\in (8/3,3)$, then $\mathcal{N}_r^+,\mathcal{N}_r^0,\mathcal{N}_r^-$ are non-empty and $I_r<0$ for all $r>0$. \smallskip
		\item[vi)] If $p\in (3,10/3)$, then there exist $0<r^*<r_0^*$ such that  \smallskip
		\begin{enumerate}
			\item[1)] $\mathcal{N}_r^+,\mathcal{N}_r^-$ are non-empty if, and only if $r>r^*$. \smallskip
			\item[2)] $\mathcal{N}_r^0\neq \emptyset$ if, and only if $r\ge r^*$.
		\end{enumerate}
		Moreover if $r>r_0^*$, then $I_r=\inf_{S_{r}}E<0$  while if $r\in[r^*,r_0^*]$, then $I_r\ge 0$.
	\end{enumerate}
\end{theorem}
Theorem \ref{Nehari} (parts of it)  can be found in most of the works cited until now,  in particular we would like to refer the reader to the works \cite{jeanluo,cattoetal}, where some calculations can be found explicitly.

\medskip

Our contribution here is to show how with a general framework it is possible
to  connect all these results with the partitioning of the Nehari set $\mathcal{N}_r$ in terms of $\mathcal{N}_r^+,\mathcal{N}_r^0,\mathcal{N}_r^-$. Moreover we give
a characterisation of the quantities $r^*,r_0^*$  in terms of $R_{p}$, namely
\begin{equation*}
r^*=\left(4\frac{10-3p}{3p-8}\right)^{\frac{3p-10}{4(p-3)}}\left(\frac{4p}{3(p-2)(3p-8)}\right)^{\frac{1}{2(p-3)}}\inf_{w\in S_1}R_p(w),
\end{equation*}
and 
\begin{equation*}
r_0^*=\left(\frac{2(10-3p)}{3p-8}\right)^{\frac{3p-10}{4(p-3)}}\left(\frac{p}{3p-8}\right)^{\frac{1}{2(p-3)}}\inf_{w\in S_1}R_p(w),
\end{equation*}
(see \eqref{eq:rstar}) and it will be evident that they are related to some geometrical properties of the Nehari set (see Proposition \ref{fiberingmaps}).
Note that (as already known) for $p\in[10/3,6)$ we have $I_r=-\infty$, which is equivalent to $\mathcal{N}_r=\mathcal{N}_r^-$ and also suggests a 
mountain pass geometry (see Bellazini et al. \cite{JBL}). Observe  that itens $iv)$ and $vi)$ of Theorem \ref{Nehari} give also results of 
non-existence of critical points of $E$ over $S_r$
depending on $r$. Indeed since every critical point of $E$ constrained to $S_r$ belongs to $\mathcal{N}_r$, it follows that if $\mathcal{N}_r$ is empty, then there is no critical point at all; therefore as an immediate consequence
of Theorem \ref{Nehari} we infer
\begin{corollary}\label{non} 
	The functional $E$ constrained to $S_r$ has no critical points if:
	\begin{enumerate}
		\item[i)] $p=10/3$ and $\frac{5}{3K_{\emph{GN}}}\frac{1}{\lambda}\geq r^{{2}/{3}}$,
		% where $K_{\emph{GN}}>0$ is a suitable constant, 
		\smallskip
		\item[ii)]  $p\in (3,10/3)$ and $r<r^*$.
	\end{enumerate}
\end{corollary} 

Note that the results in Theorem \ref{Nehari} and Corollary \ref{non} are independent on $q>0$ and the unique case in which $\lambda$
has a role is when $p=10/3$.

\medskip

%\subsection{The case $p=3$}
\noindent{\bf The case $p=3$. } Here
the situation changes drastically in the sense that $r$ has not anymore a role and the properties 
of  the Nehari sets depends on $q$ and $\lambda$.
In order to make clear this dependence, we use the notations 
$\mathcal{N}_{q,\lambda}, I_{q,\lambda}, \ldots$ instead of the previous
$\mathcal{N}_r, I_r\ldots$. 

We prove  the following.

\begin{theorem}\label{Nehari1} 
Let $p=3$ and $r>0$. For each fixed $q>0$, there exist positive constants $\lambda_q^*<\lambda_{0,q}^*$ such that 
	\begin{enumerate}
		\item[i)]  $\mathcal{N}_{q,\lambda}^+,\mathcal{N}_{q,\lambda}^-$ are non-empty if, and only if $\lambda>\lambda_q^*$. Moreover, if $\lambda>\lambda^*_q$ then $\mathcal{N}_{q,\lambda}^0\neq\emptyset$. \smallskip
		\item[ii)] If $\lambda>\lambda_{0,q}^*$, then $I_{q,\lambda}=\inf_{S_{r}}E<0$, while if $\lambda\in(0,\lambda_{0,q}^*)$, then $I_{q,\lambda}\ge 0$.
	\end{enumerate}	
\end{theorem}
Similarly to the quantities $r_0^*,r^*$, the quantities $\lambda_{0,q}^*,\lambda_{q}^*$ have a geometrical interpretation and are given by
\begin{equation*}
\lambda_{0,q}^*=\left(\frac{9}{2}\right)^{{1}/{2}}q^{{1}/{2}}\inf_{w\in S_1}\frac{\left(\displaystyle\int|\nabla w|^2\displaystyle\int\phi_ww^2\right)^{{1}/{2}}}{\displaystyle\int|w|^3},
\end{equation*}
and
\begin{equation*}
\lambda_q^*=2q^{{1}/{2}}\inf_{w\in S_1}\frac{\left(\displaystyle\int|\nabla w|^2\displaystyle\int\phi_ww^2\right)^{{1}/{2}}}{\displaystyle\int|w|^3}.
\end{equation*}
We observe that unlike Theorem \ref{Nehari} item $vi)$, in Theorem  \ref{Nehari1} we were not able to describe the behavior of $\mathcal{N}_{q,\lambda_{0,q}^*}$. This is due to the fact that $u\in \mathcal{N}_{q,\lambda_{0,q}^*}$ if, and only if $u$ is a minimiser of the quotient
\begin{equation*}
\frac{\left(\displaystyle\int|\nabla w|^2\displaystyle\int\phi_ww^2\right)^{{1}/{2}}}{\displaystyle\int|w|^3}
\quad \text{ on } \ S_{1}.
\end{equation*}
The fact that the above quotient %$\frac{\left(\int|\nabla w|^2\int\phi_ww^2\right)^{\frac{1}{2}}}{\int|w|^3}$
 is bounded away from zero is due to Lions \cite{lions}, however, it is an open problem if this functional has a minimizer. As was pointed out in \cite{cattoetal}, the minimizers of that functional also are (up to some constant) global minimizers of $I_{q,\lambda_{0,q}^*}$ and  $I_{q,\lambda_{0,q}^*}=0$. As in before, we can deduce by using Theorem \ref{Nehari1} a non-existence result.
\begin{corollary}  
Let $p=3$ and  $r,q>0$. The functional $E$ constrained to $S_r$ has no critical points if $\lambda\in(0,\lambda_q^*)$. 
\end{corollary}

\bigskip

%\noindent {\bf (2) Estimate on $r$ for the existence of global minimisers when $p\in(2,3)$.}

(III)  Our third type of result complements \cite[Theorem 4.1]{bellasici1}.
Indeed in 
%We address the problem of existence of global minimizers for $E$ on $S_{r}$ when $p\in (2,3)$.
 \cite{bellasici1} (see also \cite{cattoetal}) the authors proved, among other things, that for small $r$, there exist minimizers for $E$ over $S_r$.
%Observe that the existence of minimizers when $p\in (2,3)$ and $r$ is small is not new; in fact, in the work %\cite{bellasici1} (see also \cite{cattoetal}) the authors proved, among other things, that for small $r$, there exist %minimizers for $E$ over $S_r$.
With our approach we are able to give a
quantitative estimate on the ``smallness'' of $r$ in terms of $R_{p}$
%radius of the sphere
 which guarantees the existence of minimisers.
%in terms of the functional $R_{p}$ defined before.

\begin{theorem}\label{thmE} Let  $p\in (2,3)$. Then 
%there exists a  constant $C_{q,\lambda,p}>0$ such that $R_p(u)\ge C_{q,\lambda,p}$ for each $u\in S_1$. %Moreover,
 for each 
$$r\in \left(0,\left[\frac{1}{2}\left(\frac{2(p-2)}{p}\right)^{\frac{2}{3p-8}}\right]^{\frac{3p-8}{4(3-p)}}\inf_{w\in S_1}R_p(w)\right)$$ there exists $u\in S_r$ satisfying $E(u)=\min_{S_{r}}E$.
\end{theorem}

\bigskip

%\noindent {\bf (3) Existence of global minimiser  with positive energy when $p\in(3,10/3)$.}

(IV) The fourth type of result deal with 
existence of global minimiser  with positive energy when $p\in(3,10/3)$
which have never been treated in the literature.
In this case the inequality $\inf_{S_{r}}E<0$ is no longer true for $r\in[r^*,r_0^*]$. 
Moreover $\inf_{S_{r}}E=0$ if $0<r\le r_0^*$ and it is not achieved.

We  extend these results by showing existence of local minimizers for $E$ on $S_{r}$ with positive energy, when $r$ belongs to a neighborhood of $r_0^*$.
Unfortunately we are not able to cover the whole range $(3,10/3)$. Our result is

\begin{theorem}\label{thmp>3} There exists $p_0\in(3,10/3)$ such that if $p\in (p_0,10/3)$, then
	\begin{enumerate}
		\item[i)] the function $[r^*,\infty)\ni r\mapsto I_r$ is decreasing, $I_{r_0^*}=0$ and $I_r>0$ for $r\in[r_*,r_0^*)$; \medskip
		\item[ii)] for each $r\in [r^*,r_0^*)$ there exists $u\in \mathcal{N}_r^+\cup \mathcal{N}_r^0$ such that $I_r=E(u)$;\medskip
		\item[iii)]  there exists $\varepsilon>0$ such that if $r\in(r_0^*-\varepsilon,r_0^*)$, then there exists $u\in \mathcal{N}_r^+$ such that $I_{r}=E(u)$.
		\end{enumerate}
\end{theorem}
Indeed we find explicitly the number $$p_{0} = \frac{73+\sqrt{145}}{27}$$
which is new in the literature.

We believe it is an interesting  problem to study the case $p\in (3,p_{0}]$.
%We conclude with some estimates; in particular the second one is new in the literature.
%\begin{theorem}\label{estima} There holds
%	\begin{enumerate}
%		\item[i)] For each $p\in(2,3)$, there exists a positive function $c_p:(0,\infty)^2\to \mathbb{R}$ such that if $r_1<r_2$, then
%		\begin{equation*}
%		I_{r_2}>\left(\frac{r_2}{r_1}\right)^3I_{r_1}+c_{p}(r_1,r_2)\lambda\left(\frac{r_1}{r_2}\right)^{p-3} \left[\left(\frac{r_1}{r_2}\right)^{2(p-3)}-1\right].
%		\end{equation*}
%		\item[ii)] If $p\in (3,10/3)$ and $r^*<r_1<r_2$, then there exists a positive constant $c_p$ such that 
%		\begin{equation*}
%		I_{r_2}<\left(\frac{r_2}{r_1}\right)^{3}I_{r_1}-\frac{c_p}{r_1}\left(\frac{r_2}{r_1}\right)^{p} \left[\left(\frac{r_2}{r_1}\right)^{2(p-3)}-1\right].
%		\end{equation*}
%	\end{enumerate}
%\end{theorem}

\medskip

\noindent {\bf Organisation of the paper.}
The paper is organised as follows.

In Section \ref{sec:A} we study deeply the Rayleigh quotient $R_{p}$.
Indeed most of the results are based on its properties.
We give then the proof of Theorem \ref{th:desigualdadenova} 
and Theorem \ref{cattoinep<3}.

In Section \ref{sec:natural} we introduce the set $\mathcal N_{r}$
and give a description
via the fibering method of  its subsets
$\mathcal N^{+}_{r}, \mathcal N_{r}^{-},\mathcal N_{r}^{0}$
on which study the energy functional $E$.
In particular we show that $\mathcal N_{r}^{+}, \mathcal N_{r}^{-}$
are differentiable manifolds and natural constraints  for $E$ (see Theorem \ref{naturalconstraints}).

In Section \ref{sec:nehari}
 we study deeply these sets 
depending on the parameters $q,\lambda,p,r$.
This analysis will allow us to  prove our second type of results,
namely Theorem \ref{Nehari} and Theorem \ref{Nehari1}.

In Section \ref{sec:subad} we show the subadditivity condition
for $I_{r}$ that will serve a prerequisite for the subsequent Section.

In Section \ref{sec:Rp} we prove
 Theorem \ref{thmE} and  Theorem \ref{thmp>3}. 

Section \ref{sec:p10/36}  is devoted to show as we can recover with our methods the results in  \cite{JBL}.

In the  final Section \ref{sec:ineq}  we give new estimate concerning $I_{r_{1}}$ and $I_{r_{2}}$
for $r_{1}<r_{2}$ obtained by means of the fibering approach.

\medskip

\noindent {\bf Notations.}
As a matter of notations, in all the paper we denote with $\|\cdot \|_{p}$ the $L^{p}-$norm
in $\mathbb R^{3}$. We use $o_{n}(1)$ to denote a vanishing sequence.
In all the paper, given a function $u$ and $t>0$, we set
$$u^t(x)=t^{\frac{3}{2}}u(tx).$$
Note that $\|u\|_{2} = \|u^{t}\|_{2}.$

\section{The nonlinear Rayleigh quotient $R_{p}$}\label{sec:A}
Let us start  with a simple and general result,
whose proof is straightforward so is omitted.
\begin{proposition}\label{systemsolve}  Suppose that
	$b\neq 0$, $ce-bf\neq 0$, $(bd-ae)/(ce-bf)>0$, $(af-cd)/(ce-bf)>0$, $A,B,C>0$ and $p\in(2,6)\setminus\{3\}$. Then the system
	\begin{equation}\label{systemfiber}
	\left \{ \begin{array}{ll}
	atA+brB+cr^{\frac{p}{2}-1}t^{\frac{3p}{2}-5}C =0, \medskip \\
	dtA+erB+fr^{\frac{p}{2}-1}t^{\frac{3p}{2}-5}C =0.
	\end{array}\right.
	\end{equation}
	admits a unique solution $r,t>0$. Moreover explicitly we have
	\begin{equation*}
	r=\left(\frac{bd-ae}{ce-bf}\right)^{\frac{1}{2(p-3)}}	\left(\frac{af-cd}{ce-bf}\right)^{\frac{3p-10}{4(p-3)}}\frac{A^{\frac{3p-8}{4(p-3)}}B^{\frac{3p-10}{4(p-3)}}}{C^{\frac{1}{2(p-2)}}}.
	\end{equation*}
\end{proposition}

\medskip

Recall the next two results.
\begin{lemma}[Catto et al. \cite{cattoetal}] \label{catto}For each $p\in (2,6)$ and $r>0$, there exist a sequence of functions $\{u_n\}\subset S_r$ and positive constants $C_1,C_2$ and $C_3$, satisfying
	\begin{equation*}
	\int |u_n|^p=C_1,\ \ \int |\nabla u_n|^2=C_2n^{\frac{2}{3}},\ \ \int \phi_{u_n}u_n^2\le \frac{C_3}{n^{\frac{2}{3}}},\quad \forall n\in\mathbb{N}.
	\end{equation*}
\end{lemma}
\begin{lemma}[Catto et al. \cite{cattoetal}]\label{ineqcatto} Assume that 	$p\in [{8}/{3},3]$, then there exist a constant $C>0$ such that
	\begin{equation}\label{ineqcattooo}
	\int |u|^p\le C\left(\int u^2\right)^{2(3-p)}\left(\int \phi_uu^2\right)^{\frac{p-2}{2}}
	\left(\int |\nabla u|^2\right)^{p-2},\quad \forall u\in H^1(\mathbb{R}^3).
	\end{equation}
	If $p\in [3,{10}/{3}]$, then there exist a constant $C>0$ such that
	\begin{equation}\label{ineqcattooo1}
	\int |u|^p\le C\left(\int u^2\right)^{2(p-3)}\left(\int \phi_uu^2\right)^{\frac{10-3p}{2}}\left(\int |\nabla u|^2\right)^{\frac{3p-8}{2}},\quad\forall u\in H^1(\mathbb{R}^3).
	\end{equation}
\end{lemma}

\medskip

Let us define, for $p\in (2,6)\setminus\{3\}$ the quotient
\begin{equation}\label{NRQ}
R_p(u)=\frac{\left(\displaystyle\int|\nabla u|^2\right)^{\frac{3p-8}{4(p-3)}}\left(q\displaystyle\int\phi_uu^2\right)^{\frac{10-3p}{4(p-3)}}}{\left(\lambda\displaystyle\int|u|^p\right)^{\frac{1}{2(p-3)}}},
\quad  u\in S_{1}.
\end{equation}

Note that in particular
\begin{equation}\label{eq:R8/3}
R_{8/3}(u) =\frac{\left(\lambda\displaystyle \int |u|^{8/3} \right)^{3/2}}{\left( q\displaystyle\int \phi_{u} u^{2}\right)^{3/2} }.
\end{equation}

%The next result concerns indeed  the slight larger case $p\in(2,10/3)\setminus\{3\}$.
%We will need it in the form expressed  below, although
The next result is %essentially known since the statements are 
just  Lemma \ref{catto} and  the inequality \eqref{ineqcattooo1} of Lemma \ref{ineqcatto}  rewritten 
in terms of our  functional $R_{p}$. 
%For this reason the statements are indeed also true for
%$p=3$ and $p=10/3$. However we not consider them here since 
%the case $p=10/3$ has been already studied, and the case $p=3$
%will be treated later.

\begin{proposition}\label{NQRP} The functional $R_p$ defined is continuous. Moreover:
	\begin{enumerate}
		\item[i)] if $p\in (2,3)$, then the functional $R_p$ is unbounded from above, \medskip
		\item[ii)] if $p\in (3,{10}/{3}]$, then the functional $R_p$ is bounded away from $0$.
	\end{enumerate}
\end{proposition}	
\begin{proof} The continuity is obvious. To prove  $i)$, if  $\{u_{n}\}$ is the sequence given in Lemma \ref{catto},
	since $p<3$,  it follows that $R_p(u_n)\ge Cn$ 
	%for each $n\in\mathbb{N}$, 
	where $C$ is some positive constant. The proof of $ii)$ is a direct consequence of \eqref{ineqcattooo1} of Lemma \ref{ineqcatto}. 
\end{proof}

%\begin{remark} %Note that $ii)$ is indeed known since it is equivalent to the inequality \eqref{ineqcattooo1} 
%whenever $u\in S_{1}$. %is equivalent to the boundedness from below of $R_p$ on $S_{1}$. 
%	We will prove in Proposition \ref{boundedbelow}  by our  new approach that $R_p$ is indeed bounded from below also when 
%$p\in(2,3)$.
%\end{remark}

 For future reference we   consider  the system
\begin{equation}\label{decre}
\left\{
\begin{aligned}
rt^2\int |\nabla u|^2+\frac{r^2t}{4}q\int \phi_uu^2-\frac{3(p-2)}{2p}r^{{p}/{2}}t^{\frac{3(p-2)}{2}}\lambda\int |u|^p=0, \\
\frac{q}{2}r^2t\int \phi_uu^2-\frac{p-2}{p}r^{{p}/{2}}t^{\frac{3(p-2)}{2}}\lambda\int |u|^p=0, 
\end{aligned}
\right.
\end{equation}
where $u\in S_1$ and $r,t>0$. From Proposition \ref{systemfiber} and recalling $R_{p}$ defined  in \eqref{NRQ} 
and \eqref{eq:R8/3}, we have that if $p\in(2,6)\setminus\{3\}$, then the system has a unique solution $(\widetilde{r}(u),\widetilde{t}(u))$ which is given by  \medskip

\noindent {\bf $\bullet$ } if $p\in (2,3)$ with $p\neq 8/3$:
%, the system has a unique solution $(\widetilde{r}(u),\widetilde{t}(u))$ which is given by 
\begin{equation}\label{eq:no8/3}
\widetilde{r}(u)=\left[\frac{1}{2}\left(\frac{2(p-2)}{p}\right)^{\frac{2}{3p-8}}\right]^{\frac{3p-8}{4(3-p)}}R_p(u),\quad
\widetilde{t}(u)=\left(\frac{p}{2(p-2)}\widetilde r(u)^{\frac{4-p}{2}}\frac{q}{\lambda}\frac{\displaystyle\int \phi_uu^2}{\displaystyle\int |u|^p}\right)^{\frac{2}{3p-8}}; \medskip
\end{equation}

\noindent {\bf $\bullet$ } if $p=8/3$:
%, the system has a unique solution $(\widetilde{t}(u),\widetilde{r}(u))$ which is given by 

\begin{equation}\label{eq:8/3}
\widetilde{r}(u)=\frac{1}{2^{\frac{3}{2}}}R_{8/3}(u), \qquad
\widetilde{t}(u)=\frac{1}{2^{\frac{5}{2}}}\frac{\left(\lambda\displaystyle \int|u|^p\right)^{\frac{3}{2}}}{\left(q \displaystyle\int\phi_uu^2\right)^{\frac{1}{2}}\displaystyle\int |\nabla u|^2}.
\end{equation}

\begin{remark}\label{rem:constantecontinua}
Note that $\widetilde r(u)$ as a function of $p$ is continuous in $p=8/3$
	since
	$$\lim_{p\to 8/3} \left[\frac{1}{2}\left(\frac{2(p-2)}{p}\right)^{\frac{2}{3p-8}}\right]^{\frac{3p-8}{4(3-p)}}= \frac{1}{2^{\frac{3}{2}}}.$$
\end{remark}

We recall the following Hardy-Littlewood-Sobolev inequality (see \cite[Theorem 
4.3]{liebloss}):
\begin{theorem}\label{HLS}
	Assume that $1<a,b<\infty$ satisfies
	\begin{equation*}
	\frac{1}{a}+\frac{1}{b}=\frac{5}{3}.
	\end{equation*}
	Then there exists a constant $H_{a,b}>0$ sucht that 
	\begin{equation*}
	\left|\int\int_{\mathbb{R}^3\times\mathbb{R}^3}\frac{f(x)g(y)}{|x-y|}dxdy\right|\le H_{a,b}\|f\|_a\|g\|_b, \quad \forall f\in L^a(\mathbb{R}^3), g\in L^b(\mathbb{R}^3).
	\end{equation*}
	%\textcolor{red}{Moreover, there exists a positive constant $C$ such that
	%\begin{equation*}
	%H_{a,b}\le C\left[\left(\frac{1}{1-1/a}\right)^{{1}/{3}}+\left(\frac{1}{1-1/b}\right)^{{1}/{3}}\right]. 
	%\end{equation*}}
\end{theorem}

Then we can prove the following result.
\begin{proposition} \label{prop:desigualdadenova}
For each $p\in [10/3,6)$, there exists a constant $C_{q,\lambda,p}>0$
such that
	\begin{equation*}
	R_p(u)\ge C_{q,\lambda,p}, \quad \forall u\in S_1,
	\end{equation*}
	%	Moreover $\overline{r}(u)=r(u)$ when $p=10/3$.
\end{proposition}
\begin{proof} We have by definition
	\begin{equation*}
	R_p(u)=\frac{\left(\displaystyle\int|\nabla u|^2\right)^{\frac{3p-8}{4(p-3)}}}{\left(\lambda\displaystyle\int|u|^p\right)^{\frac{1}{2(p-3)}}\left(q\displaystyle\int\phi_uu^2\right)^{\frac{3p-10}{4(p-3)}}},
	\quad u\in S_{1}.
	\end{equation*}
We can assume that $\|\nabla u\|_2=1$ and hence the conclusion is a simple consequence of Sobolev embbedings 
	and the Hardy-Littlewood-Sobolev inequality.
\end{proof}

% LEMMA SOBRE a,b
%The following lemma is straightforward.
%\begin{lemma}\label{inequalities} Assume that $a\in (6/5,3/2)$ and let $b=3a/(5a-3)$.
%	Then 
%	$$12/5<2a<3 \quad \text{and} \quad 2<2b<2a<3.$$
%\end{lemma}
More involved is the proof of the next result.

\begin{proposition}\label{th:boundedbelow} For each $p\in (2,3)$, there exists 
	a constant $C_{q,\lambda,p}>0$  such that
	\begin{equation*}
	R_p(w)\ge C_{q,\lambda,p}, \quad \forall w\in S_1,
	\end{equation*}
\end{proposition}
\begin{proof}
	%	\subsection*{Proof of Theorem \ref{th:boundedbelow}}
	First note  that
	\begin{equation*}
	R_p(w^t)=R_p(w) \quad \text{ and } \quad \displaystyle\int |\nabla w^t|^2=t^2\displaystyle\int |\nabla w|^2
	\quad  \forall w\in S_1,\ t>0.
	\end{equation*}
	From this it follows that 
	\begin{equation}\label{eq:11}
	\inf_{w\in S_1}R_p(w)=\inf\left\{R_p(w):w\in S_1,\ \int |\nabla w|^2=1\right\}.
	\end{equation}
	Indeed, for any $\varepsilon>0$ there exists $u\in S_{1}$ such that $R_{p}(u)\leq \inf_{S_{1}} R_{p} +\varepsilon$.
	Then, if we consider $u^{t_{*}}$, where $t_{*} \displaystyle\int|\nabla u|^{2}=1 $, we have that
	$$u^{t_{*}}\in S^{1} \ \text{ and }\ \int|\nabla u^{t_{*}}|^{2}=1.$$
	Consequently 
	$$\inf\left\{R_p(w):w\in S_1,\ \int |\nabla w|^2=1\right\}\leq R_{p}(u^{t_{*}}) = R_{p}(u)\leq \inf_{ S_{1}} R_{p}+\varepsilon $$
	and \eqref{eq:11} follows.
	%	Since $\displaystyle\int |\nabla w^t|^2=t^2\displaystyle\int |\nabla w|^2$ for all $ w\in S_1$ and $t>0$ we conclude that
	%	\begin{equation*}
	%	\inf_{w\in S_1}R_p(w)=\inf\left\{R_p(w):w\in S_1,\ \int |\nabla w|^2=1\right\}.
	%	\end{equation*}
	The approach to prove the theorem will be different according to the values of $p$.
	\medskip
	
	{\bf Case 1: $p\in(8/3,3)$.}
	\vskip.3cm
	Assume on the contrary that there exists a sequence $\{w_n\}\subset S_1$ such that $R_p(w_n)\to 0$ as $n\to +\infty$. 
	Let  $\widetilde{r}(w_n)$ and 
	$\widetilde{t}(w_n)$ the solutions of system \eqref{decre}, see \eqref{eq:no8/3}, and
	set for brevity
	$r_n=\widetilde{r}(w_n)$ and $u_n= r_n^{{1}/{2}}w_n^{\widetilde{t}(w_n)}$. 
%	From the definition of $\widetilde{r}(w_n)$ and 
	%$\widetilde{t}(w_n)$  (see \eqref{eq:no8/3})
	It is easy to see  that, for all $n\in \mathbb N$, 
	\begin{equation}\label{BG}
	\left\{
	\begin{aligned}
	\int |\nabla u_n|^2+\frac{q}{4}\int \phi_{u_n}u_n^2-\frac{3(p-2)}{2p}\lambda\int |u_n|^p=0, \\
	\frac{q}{2}\int \phi_{u_n}u_n^2-\frac{p-2}{p}\lambda\int |u_n|^p=0.
	\end{aligned}
	\right. %\ \forall n\in\mathbb{N}.
	\end{equation}
	Now observe that \eqref{BG} is the same as \cite[equation (4.9)]{bellasici1} and therefore 
	\begin{equation*}
	E(u_n)=\frac{3-p}{2-p}\int |\nabla u_n|^2<0,\quad \forall n\in \mathbb{N},
	\end{equation*}
	which implies that $I_{r_n}\le E(u_n)<0$ and hence, since $I_{r_n}\to 0$ as $n\to +\infty$, we obtain that $E(u_n)\to 0$ as $n\to \infty$. The last convergence implies  \cite[formula (4.10)]{bellasici1}. Therefore, by following the proof of Theorem 4.1., step  5, case (e) of \cite{bellasici1}, we reach a contradiction and hence $R_p$ is bounded from below over the sphere $S_1$.
	
	\medskip
	
	To treat the other cases of $p$, we observe that, since $p<3$, the Lemma is proved once we show that $R_{p}^{2(p-3)}$ is bounded above if 
	$\|w\|_{2} = \|\nabla w\|_{2}=1$.
	
	\medskip
	
	%\begin{equation*}
	%\inf\left\{R_p(w):w\in S_1,\ \int |\nabla w|^2=1\right\}=\sup\left\{[R_p(w)]^{2(p-3)}:w\in S_1,\ \int |\nabla w|%^2=1\right\},
	%	\end{equation*}
	%	therefore we will prove that $[R_p(w)]^{2(p-3)}$ is bounded from above if $\|w\|_2=\|\nabla w\|_2=1$.
	%In particular, by interpolation and the Sobolev inequality, it follows that
	%$$\|w\|_{2} = \|\nabla w\|_{2}=1\Longrightarrow \|w\|_{p}\leq S^{3(p-2)/2p}, \ \text{where } \ 
	%S=\sup_{w\in H^{1}(\mathbb R^{3})} \frac{\|w\|_{6}}{\|\nabla w\|_{2}}.$$	
	
	{\bf Case 2: $p\in(12/5,8/3]$.}
	\vskip.3cm
	By choosing $a=p/2$ and $b=3p/(5p-6)$ 
	%\textcolor{red}{we conclude from Lemma \ref{inequalities} that $a\in (6/5,3/2)$ and $b>1$, $2<2a<3$ and $2<2b<p$. Therefore}
	from Theorem \ref{HLS} we obtain
	\begin{equation*}
	\int \phi_ww^2\le H_{a,b}\|w^2\|_{p/2}\|w^2\|_{b}=H_{a,b}\|w\|_p^2\|w\|_{2b}^2,\quad \forall w\in H^1(\mathbb{R}^3).
	\end{equation*}
	Since 	$2<2b<p$, from the interpolation inequality we have that
	\begin{equation*}
	\|w\|_{2b}\le \|w\|_p^{\frac{2(3-p)}{3(p-2)}}\|w\|_2^{\frac{p}{3(p-2)}},
	\end{equation*}
	and hence
	\begin{equation*}
	\int \phi_ww^2\le H_{a,b}\|w\|_p^{\frac{2p}{3(p-2)}}\|w\|_2^\frac{2p}{3(p-2)}.
	%\quad \forall w\in S_1,\ \int |\nabla w|^2=1.
	\end{equation*}
	Consequently, for a suitable constant $C_{p,q}>0$ depending only on $p$ and $q$, we get 
	\begin{eqnarray*}
		R_{p}(w)^{2(p-3)}&=&\frac{\left(\displaystyle\int|\nabla w|^2\right)^{\frac{3p-8}{2}}\left(q\displaystyle\int\phi_ww^2\right)^{\frac{10-3p}{2}}}{\lambda\displaystyle\int|w|^p}\\
		&\leq& \frac{C_{q,p}}{\lambda}\frac{\left(\displaystyle\int |w|^p\right)^{\frac{10-3p}{3(p-2)}}}{\displaystyle\int |w|^p} \\
		&= &\frac{C_{q,p}}{\lambda} \left(\int|w|^p\right)^{\frac{2(8-3p)}{3(p-2)}} \\
		&\le & 2\frac{C_{q,p}}{\lambda},\qquad \forall w\in S_1, \int|\nabla w|^2=1.
	\end{eqnarray*}
	
	\medskip
	
	{\bf Case 3: $p\in(2,12/5]$.}
	\vskip.3cm
	We choose $a=b=6/5$ in Theorem \ref{HLS} and use the interpolation inequality to conclude that 
	\begin{equation*}
	\int\phi_ww^2\le H_{6/5,6/5}\|w\|_{12/5}^4\le H_{6/5,6/5}\|w\|_p^{\frac{6p}{6-p}}\|w\|_6^{\frac{2(12-5p)}{6-p}} , 
	\quad\forall w\in H^1(\mathbb{R}^3).
	\end{equation*}
	From the Sobolev inequality we obtain that, for a suitable constant $S>0$,
	depending only on $p$, we have 
	\begin{equation*}
	\int\phi_ww^2\le H_{6/5,6/5}S\|w\|_p^{\frac{6p}{6-p}} , \qquad \forall w\in S_1,\ \int |\nabla w|^2=1.
	\end{equation*}
	Consequently, for a suitable constant $C_{p,q}>0$ depending only on $p$ and $q$, we have
	\begin{eqnarray*}
		R_{p}^{2(p-3)} (w) &=&\frac{\left(\displaystyle\int|\nabla w|^2\right)^{\frac{3p-8}{2}}\left(q\displaystyle\int\phi_ww^2\right)^{\frac{10-3p}{2}}}{\lambda\displaystyle\int|w|^p}\\
		&\leq& \frac{C_{q,p}}{\lambda}\frac{\left(\displaystyle\int |w|^p\right)^{\frac{3(10-3p)}{6-p}}}{\displaystyle\int |w|^p}  \\
		&= &  \frac{C_{q,p}}{\lambda}\left(\int |w|^p\right)^{\frac{8(3-p)}{6-p}} \\
		&\le & 2 \frac{C_{q,p}}{\lambda},\quad \ \forall w\in S_1, \int|\nabla w|^2=1,
	\end{eqnarray*}
	and hence the proof is concluded. 
\end{proof}

As a consequence of the previous proposition, we have  the  new inequalities
stated in Theorem \ref{th:desigualdadenova} and Theorem \ref{cattoinep<3}.
%It follows a new inequality
%.Also as a byproduct of Theorem \ref{thmE} we have the following new inequality 
%which never appeared in the literature.
\subsection{Proof of Theorem \ref{th:desigualdadenova} 
and Theorem \ref{cattoinep<3}}

	They follows respectively by
 Proposition \ref{prop:desigualdadenova} and
 Proposition \ref{th:boundedbelow} with a simple $L^{2}-$normalization.
	
%	Analogously, from Proposition \ref{th:boundedbelow} there exists $C_{q,\lambda,p}>0$ such that
%	\begin{equation*}
%	\frac{\left(\displaystyle\int|\nabla u|^2\right)^{\frac{3p-8}{4(p-3)}}\left(q\displaystyle\int\phi_uu^2\right)^{\frac{10-3p}%{4(p-3)}}}{\left(\lambda\displaystyle\int|u|^p\right)^{\frac{1}{2(p-3)}}}\ge C_{q,\lambda,p}, \quad u\in S_{1}
%	\end{equation*}
%	and a simple computation shows that this is equivalent  to the desired estimate
%	in Theorem \ref{cattoinep<3}.

\section{Natural constraints for $E$}\label{sec:natural}

In this Section we  prove the existence of a natural constraint for the energy functional $E$ restricted to $S_r$. Although such a constraint 
appeared already in \cite{JBL, jeanluo}, the proof that it is a manifold and a natural constraint seems to be new.

We start with a well know Pohozaev identity, which will be quite useful: for $a,b,c,d\in\mathbb{R}$ consider the equation
\begin{equation}\label{poho}
\left\{
\begin{aligned}
-a\Delta u+bu&+c\phi_u u+d|u|^{p-2}u=0, \\
                        &u\in H^1(\mathbb{R}^3),
\end{aligned}
\right.
\end{equation}
and define $P: H^1(\mathbb{R}^3)\to \mathbb{R}$ by
\begin{equation*}
P(u)=\frac{a}{2}\int |\nabla u|^2+\frac{3b}{2}\int u^2+\frac{5c}{4}\int \phi_uu^2+\frac{3d}{p}\int |u|^p.
\end{equation*}
Then we have the following Pohozaev's identity, see \cite[Theorem 2.2]{ruiz}:
\begin{proposition}\label{pohozaev} If $u$ satisfies \eqref{poho}, then $P(u)=0$.
\end{proposition}
%The motivation behind the natural constraint $\mathcal N_{r}$ is given by
As a consequence  we get the next  result  which is already known (see \cite[Lemma 2.1]{jeanluo}) 
but that we prove  for completeness.
\begin{proposition} \label{neh} Assume that $u\in S_r$ is a critical point of $E$ restricted to $S_r$, then 
	\begin{equation*}
	\int |\nabla u|^2+\frac{q}{4}\int \phi_uu^2-\frac{3(p-2)}{2p}\lambda\int |u|^p=0.
	\end{equation*}
%	In other words $u\in \mathcal N_{r}.$
\end{proposition}
\begin{proof} Indeed, from the  Lagrange multiplier rule there exist $\mu\in\mathbb{R}$ such that
$E'(u)=\mu u$, that is, $u$ is a solution of
\begin{equation*}%\label{eq:}
-\Delta u +q\phi_{u}u -\lambda|u|^{p-2}u = \mu u.
\end{equation*}
In particular $u$ satisfies
$$\int|\nabla u|^{2}+q\int\phi_{u}u^{2} -\lambda\int|u|^{p} = \mu \int u^{2}$$
and by Proposition \ref{pohozaev} also
$$\frac12\int|\nabla u| ^{2} -\frac32\mu\int u^{2} +\frac54 q\int\phi_{u}u^{2}-\frac{3\lambda}{p}\int|u|^{p}=0$$
which together give the desired equality.	
%	From Proposition \ref{pohozaev} we obtain that
%	\begin{equation*}
%	\frac{3}{2}E'(u)u-P(u)=0,
%	\end{equation*}	
%	which implies the desired equality.
\end{proof}

Proposition \ref{neh} justifies the introduction of the set
\begin{equation}\label{eq:Nehariset}
\mathcal{N}_{r,q,\lambda}:=\left\{u\in S_r:\int |\nabla u|^2+\frac{q}{4}\int \phi_uu^2-\frac{3(p-2)}{2p}\lambda\int |u|^p=0\right\},
\end{equation}
since it contains any solution $u$ of \eqref{eq:problem}.
In the following we will simply write $\mathcal N_{r}$.
Define also
 \begin{eqnarray}%\label{eq:}
 \mathcal{N}_r^+&=&\left\{u\in \mathcal{N}_r:\int |\nabla u|^2-\frac{3(p-2)(3p-8)}{4p}\lambda\int |u|^p>0\right\}, 
 \label{eq:N+}\\
 \mathcal{N}_r^0&=&\left\{u\in \mathcal{N}_r:\int |\nabla u|^2-\frac{3(p-2)(3p-8)}{4p}\lambda\int |u|^p=0\right\},
 \label{eq:N0}\\
\mathcal{N}_r^-&=&\left\{u\in \mathcal{N}_r:\int |\nabla u|^2-\frac{3(p-2)(3p-8)}{4p}\lambda\int |u|^p<0\right\}
\label{eq:N-}
 \end{eqnarray}
Just in Subsection \ref{subsec:p=3} it will be more convenient
to explicit the dependence on $q$ and $\lambda$, instead of $r$,
since they will  have an important role.

\medskip

To obtain  basic estimates for the elements of $\mathcal N_{r}$ let us
 recall the  Gagliardo-Nirenberg inequality:
	\begin{equation}\label{GN}
	\int |u|^p\le  K_{\textrm{GN}} \left(\int |\nabla u|^2\right)^{\frac{3(p-2)}{4}}\left(\int u^2\right)^{\frac{6-p}{4}}, 
	\quad \forall u\in H^1(\mathbb{R}^3),
	\end{equation}
		where $K_{\textrm{GN}}>0$, hereafter, is the Gagliardo-Nirenberg constant which depends only on $p$. 
		Then we have 
\begin{proposition}\label{neharibounded} Let  $r,\lambda>0$ and $u\in \mathcal N_{r}$.
\begin{itemize}
\item[1.] For $p\in(2,{10}/{3})$, we have
%there exists a constant $c_{p,\lambda}>0$ such that 
	\begin{equation*}
	\int |\nabla u|^2\le  K_{\emph{GN}}^{\frac{4}{(10-3p)}} \left( \frac{3(p-2)\lambda}{2p}\right)^{\frac{4}{10-3p}}r^{\frac{6-p}{10-3p}}.
%	, \quad \forall u\in\mathcal{N}_r.
	\end{equation*}
	\item[2.] For $p=10/3$ we have
$$\frac{5}{3 \lambda K_{\emph{GN}}}\leq r^{\frac23}.%, \quad  \forall u\in\mathcal{N}_r.
$$
\item[3.] If $p\in(3,{10}/{3})$, then there exists constants $c_p, c'_{p}>0$ such that 
	\begin{equation*}
	\int |\nabla u|^2\ge \frac{c_p}{r}  \quad \text{and} \quad\int |u|^p\ge  \frac{c'_p}{\lambda r}.%, \quad  \forall u\in\mathcal{N}_r.
	\end{equation*}
	\item[4.] For $p\in({10}/{3},6)$, we have
	%there exists a constant $c_{p,\lambda}>0$ such that 
	\begin{equation*}
	\int |\nabla u|^2\ge  K_{\emph{GN}}^{\frac{4}{(10-3p)}} \left( \frac{3(p-2)\lambda}{2p}\right)^{\frac{4}{10-3p}}r^{\frac{6-p}{10-3p}}.
	%	, \quad \forall u\in\mathcal{N}_r.
	\end{equation*}
\end{itemize} 
	
\end{proposition}
\begin{proof} 
Preliminarily observe that for any $u\in \mathcal N_{r}$ we have
	\begin{equation}\label{NP1}
	\int|\nabla u|^2 \leq %+\frac{q}{4}\int \phi_uu^2-
	\frac{3(p-2)}{2p}\lambda\int |u|^p.
	\end{equation}
	Combining \eqref{NP1}  with the Gagliardo-Nirenberg inequality
 \eqref{GN} we infer, for any $u\in \mathcal N_{r}$,
	\begin{equation*}
	\int|\nabla u|^2\le K_{\textrm{GN}}\frac{3(p-2)\lambda}{2p}\left(\int |\nabla u|^2\right)^{\frac{3(p-2)}{4}}r^{\frac{6-p}{4}}.
	\end{equation*}
	From this we deduce 1., 2. and 4.
	%implies that
	%\begin{equation*}
	%\int |\nabla u|^2\le c\lambda^{\frac{4}{10-3p}}r^{\frac{6-p}{10-3p}},\quad \forall u\in\mathcal{N}_r,
	%\end{equation*}
	%which proves the first inequality. 

	\medskip
	
	3. Now assume that $p\in(3,{10}/{3})$. From  \cite[Lemma 2.3]{jeanluo}, there exist $c,c_p>0$ 
	positive constants, such that
	\begin{equation*}
	\int|\nabla u|^2+\frac{q}{4}\int \phi_uu^2-\frac{3(p-2)}{2p}\lambda\int |u|^p\ge  c\int |\nabla u|^2-c_p\left(\int |\nabla u|^2\right)^{\frac32}r^{\frac12} , \quad \forall u\in S_r.
	\end{equation*}
	Therefore 
	\begin{equation*}
	c\int |\nabla u|^2-c_p\left(\int |\nabla u|^2\right)^{\frac32}r^{\frac12}\le 0,\quad \forall u\in\mathcal{N}_r,
	\end{equation*}
	and hence
	\begin{equation}\label{eq:ccp}
	\int |\nabla u|^2\ge \left( \frac{c}{c_p} \right)^{2} \frac{1}{r}, \quad \forall u\in\mathcal{N}_r.
	\end{equation}
	As for the second estimate, it follows by \eqref{NP1} and \eqref{eq:ccp}.	
%	since for any $u\in \mathcal N_{r}$
%	\begin{equation*}
%	\int|\nabla u|^2<\frac{3(p-2)}{2p}\lambda\int |u|^p,
%	\end{equation*}
%	using \eqref{eq:ccp} we conclude.
\end{proof}

For the sake of completeness we observe, by looking at the proof of 
\cite[Lemma 2.2 and Lemma 2.3]{jeanluo}, that the constants appearing in 3. of
Proposition \ref{neharibounded} are given explicitly by
\begin{equation}\label{eq:ccc}
c_{p}=\frac{p-3}{4-p}K_{\textrm{GN}} \left( \frac{3(p-2)(4-p) 2^{7-p} }{p}\right)^{1/(p-3)},
\quad c_{p}'=\frac{2p}{3(p-2)} \left(\frac{c}{c_{p}}\right)^{2},
 \quad  c= \frac{64\pi-1}{64\pi}
\end{equation}
and do not depend on $q$.

%\begin{remark}\label{rem:melhor}
As we will see, item 2.  in Proposition \ref{neharibounded}  will be improved in Theorem \ref{iv}.
%\end{remark}

%\begin{corollary}\label{neharibounded1} 
%\textcolor{red}{JUNTAR NA PROP ANTERIOR?}
%For each $r>0$ and $p\in(3,{10}/{3})$, there exists a constant $c'_p>0$ such that 
%	\begin{equation*}
%	\int |u|^p\ge  \frac{c_p}{\lambda r}, \quad  \forall u\in\mathcal{N}_r.
%	\end{equation*}
%\end{corollary}
%\begin{proof} Indeed if $u\in \mathcal{N}_r$, then
%	\begin{equation*}
%	\int|\nabla u|^2+\frac{q}{4}\int \phi_uu^2-\frac{3(p-2)}{2p}\lambda\int |u|^p=0,
%	\end{equation*}
%	and hence 
%	\begin{equation*}
%	\int|\nabla u|^2<\frac{3(p-2)}{2p}\lambda\int |u|^p,
%	\end{equation*}
%	which combined with Lemma \ref{neharibounded} completes the proof.
%\end{proof}

%It follows from Proposition \ref{neh} that all the critical points of $E$ restricted to $S_r$ belongs to the {\sl Nehari type set}
%\begin{equation*}
%\mathcal{N}_r:=\mathcal{N}_{r,q,\lambda}=\left\{u\in S_r:\int |\nabla u|^2+\frac{q}{4}\int \phi_uu^2-\frac{3(p-2)}{2p}\lambda\int |u|^p=0\right\}.
%\end{equation*}

\medskip

\subsection{The fibration for $\mathcal N_{r}$}\label{subsec:fibration}
We will use the fibration method of Pohozaev to study $\mathcal{N}_r$. Given $u\in S_1$, define the 
fiber map 
$$ %\varphi_{r,u}:=
\varphi_{r,q,\lambda,u}: t\in (0,\infty)\longmapsto E(r^{{1}/{2}}u^t) \in \mathbb{R}$$ 
where $u^t(x)=t^{\frac{3}{2}}u(tx)\in S_{1}$ and then $r^{1/2}u^{t} \in S_{r}$.
Also in this case, until Subsection \ref{subsec:p=3} we will write simply $\varphi_{r,u}$. Then
explicitly we have
$$\varphi_{r,u}(t) = \frac{t^{2}}{2}r \int|\nabla u|^{2}+\frac{t}{4} r^{2} q\int \phi_{u}u^{2} 
- \frac{t^{\frac{3}{2}p-3} }{p} r^{p/2}  \lambda\int|u|^{p}.$$

A simple computation gives the next
\begin{lemma} \label{fibe}%If $t>0$, then $u^t\in S_1$. 
The fiber map $\varphi_{r,u}$ is a smooth function
%\in C^2$ 
and
	\begin{equation*}
	\varphi_{r,u}'(t)=tr\int |\nabla u|^2+\frac{r^2}{4}q\int \phi_uu^2-\frac{3(p-2)}{2p}t^{\frac{3p}{2}-4}r^{{p}/{2}}\lambda\int |u|^p,
	\end{equation*}
	\begin{equation*}
		\varphi_{r,u}''(t)=r\int |\nabla u|^2-\frac{3(p-2)(3p-8)}{4p}t^{\frac{3p}{2}-5}r^{{p}/{2}}\lambda\int |u|^p.
	\end{equation*}
\end{lemma}

Then we can give a complete description of the fiber $\varphi_{r,u}.$
\begin{proposition}\label{fiberingmaps} For each $u\in S_1$ the following statements hold.
	\begin{enumerate}
		\item[I)] If $p\in(2,{8}/{3})$, then $\varphi_{r,u}$ has only one critical point at $t_r^+(u)$ which is a global minimum with $\varphi_{r,u}''(t_r^+(u))>0$. \medskip
		\item[II)] If $p={8}/{3}$, we have: \smallskip
		\begin{enumerate}
			\item[1)] if $$\frac{r^2}{4}\int \phi_uu^2-\frac{r^{{p}/{2}}}{p}\lambda\int |u|^p<0,$$ 
			then $\varphi_{r,u}$ has only one critical point at $t_r^+(u)$ which is a global minimum with $\varphi_{r,u}''(t_r^+(u))>0$; \smallskip
			\item[2)] if $$\frac{r^2}{4}\int \phi_uu^2-\frac{r^{{p}/{2}}}{p}\lambda\int |u|^p\ge0,$$
			then $\varphi_{r,u}$ is strictly increasing and has no critical points. \medskip
		\end{enumerate} 
	\item[III)] If $p\in({8}/{3},{10}/{3})$, then there are three possibilities: \smallskip
	\begin{enumerate}
		\item[1)] $\varphi_{r,u}$ has exactly two critical points at $t_r^-(u)<t_r^+(u)$. Moreover $t_r^+(u)$ corresponds to a local minimum while $t_r^-(u)$ corresponds to a local maximum with  $\varphi_{r,u}''(t_r^+(u))>0$ and $\varphi_{r,u}''(t_r^-(u))<0$; \smallskip 
		\item[2)] $\varphi_{r,u}$ is strictly increasing and has exactly one critical point at $t_r^0(u)$. Moreover $t_r^{0}(u)$ corresponds to an inflection point; \smallskip
		\item[3)]  $\varphi_{r,u}$ is strictly increasing and has no critical points.
	\end{enumerate} \medskip
\item[IV)] If $p={10}/{3}$, we have: \smallskip
\begin{enumerate}
	\item[1)] if $$\frac{r}{2}\int |\nabla u|^2-\frac{r^{{p}/{2}}}{p}\lambda\int |u|^p<0,$$
	then $\varphi_{r,u}$ has only one critical point at $t_r^-(u)$ which is a global maximum with
	 $\varphi_{r,u}''(t_r^-(u))<0$; \smallskip
\item[2)] if $$\frac{r}{2}\int |\nabla u|^2-\frac{r^{{p}/{2}}}{p}\lambda\int |u|^p\ge0,$$
 then $\varphi_{r,u}$ is strictly increasing and has no critical points.
\end{enumerate} \medskip 

	\item[V)] If $p\in({10}/{3},6)$, then $\varphi_{r,u}$ has only one critical point at $t_r^-(u)$ which is a global maximum with $\varphi_{r,u}''(t_r^-(u))<0$.
	\end{enumerate}
\end{proposition}
\begin{proof}
	It is straightforward.
\end{proof}	
A direct application of the Implicit Function Theorem shows that
\begin{lemma}\label{diffet+} Fix $u\in S_1$ and suppose that $(a,b)\ni r\mapsto t_r^+(u) \ 
	(\text{respectively } t_r^-(u))$ is well defined. Then $(a,b)\ni r\mapsto t_r^+(u) \ (\text{respectively } t_r^-(u))$ is $C^1$ in $(a,b)$.
\end{lemma}

From Lemma \ref{fibe} it is easy to see that, for each $r>0$, $\mathcal N_{r}$ given in \eqref{eq:Nehariset}
can be written also as
\begin{equation*}
\mathcal{N}_r=\left\{u\in S_1: \varphi_{r,u}'(1)=0\right\}
\end{equation*}
which, in some sense, justifies the name of Nehari set.
Moreover it holds (see \eqref{eq:N+}, \eqref{eq:N0} and \eqref{eq:N-}) that
\begin{equation}
\begin{split}\label{eq:decomposicaoN}
\mathcal{N}_r^+&=\left\{u\in \mathcal{N}_r:\varphi_{r,u}''(1)>0\right\}, \\
\mathcal{N}_r^0&=\left\{u\in \mathcal{N}_r:\varphi_{r,u}''(1)=0\right\}, \\
\mathcal{N}_r^-&=\left\{u\in \mathcal{N}_r:\varphi_{r,u}''(1)<0\right\}
\end{split}
\end{equation}
and $\mathcal{N}_r=\mathcal{N}_r^+\cup \mathcal{N}_r^0\cup \mathcal{N}_r^-$. 

\begin{remark}\label{rem:N+0-}
Note that, given $u\in S_{1}$, $t^{*}$ is a critical point of the fiber map $\varphi_{r,u}$ if and only if $r^{1/2} u^{t^{*}}\in \mathcal N_{r}$. Actually
$t^{*}$ is a minimum (respect. maximum or inflection) point of  $\varphi_{r,u}$ if and only if
$r^{1/2} u^{t^{*}}\in \mathcal N_{r}^{+}$ (respect. $\mathcal N_{r}^{-}$ or $\mathcal N_{r}^{0}$). 
\end{remark}

In the following we study deeply the sets $\mathcal{N}_r^+$ and $\mathcal{N}_r^-$.

\subsection{$\mathcal{N}_r^+$ and $\mathcal{N}_r^-$ as natural constraints} \label{sec:N+-}

Let us start by defining, for $r>0$, the functions
\begin{equation}
\begin{split}\label{hg}
h(u)&=\frac{1}{2}\int u^2-\frac{r}{2}, \qquad \text{for } \ u\in H^1(\mathbb{R}^3), \\
g(u)&=\varphi_{r,u}'(1),\qquad \text{for } \ u \in  S_{1}. 
\end{split}
\end{equation}

%\begin{equation}\label{hg}
%h(u)=\frac{1}{2}\int u^2-\frac{r}{2}\qquad  \mbox{ and }\qquad  g(u)=\varphi_{r,u}'(1),\quad u\in H^1(\mathbb{R}^3).
%\end{equation}

\begin{lemma}\label{manifold} 
%If $\mathcal{N}_r^+\neq\emptyset$ $(\mathcal{N}_r^-\neq\emptyset)$, then 
Whenever nonempty, $\mathcal{N}_r^+$  and $\mathcal{N}_r^-$ are  $C^1$ manifolds in $H^1(\mathbb{R}^3)$
of co-dimension $2$.
\end{lemma}
\begin{proof} 
Let us show the proof for $\mathcal N_{r}^{+}$ since for  $\mathcal N_{r}^{-}$
is completely analogous.

The proof will follow once we prove that $h'(u)\neq 0$, $g'(u)\neq0$ and $h'(u)$, $g'(u)$ are linearly independent for each $u\in \mathcal{N}_r^+$. In fact, $h'(u)\neq 0$ is straightforward. Suppose on the contrary that there exists $u\in \mathcal{N}_r^+$   and $c\in \mathbb{R}$ such that 
	\begin{equation*}
	g'(u)=c h'(u).
	\end{equation*}
	It follows that 
	\begin{equation*}
	-2\Delta u-cu+q\phi_uu-\frac{3(p-2)}{2}\lambda|u|^{p-2}u=0.
	\end{equation*}
	From Propostion \ref{pohozaev} we conclude that
	\begin{equation*}
	\left\{
	\begin{aligned}
	\int |\nabla u|^2-\frac{3c}{2}r+\frac{5}{4}q\int \phi_uu^2-\frac{9(p-2)}{2p}\lambda\int |u|^p=0, \\
	2\int |\nabla u|^2-cr+q\int \phi_uu^2-\frac{3(p-2)}{2}\lambda\int |u|^p=0, \\   
	\int |\nabla u|^2+\frac{q}{4}\int \phi_uu^2-\frac{3(p-2)}{2p}\lambda\int |u|^p=0.
	\end{aligned}
	\right.
	\end{equation*}
	Let us set for brevity
	 $$A=\int |\nabla u|^2\,,  \quad B=q\int \phi_uu^2\,,  \quad C=\lambda\int |u|^p$$
	  and solve the system with respect to these variables. A simple calculation shows that it has a unique solution when $p\neq 3$, in which case 
	\begin{equation*}
	A=\frac{r c(8-3p)}{8(p-3)},\quad B=\frac{r c(3p-10)}{2(p-3)},\quad C=\frac{r c p}{6(p-2)(p-3)}.
	\end{equation*}
	We substitute $A,C$ in $\varphi_{r,u}''(1)$ to conclude that 
	\begin{equation*}
\varphi_{r,u}''(1)=0,
	\end{equation*}
	and hence a contradiction. If $p=3$ we have two cases: when $c\neq 0$, then the system has no solution, which is a contradiction, however, when $c=0$, the system has the following solution
	\begin{equation*}
	A=\frac{C}{4},\quad  B=C,\quad C>0.
	\end{equation*}
	We substitute $A,C$ in $\varphi_{r,u}''(1)$ to conclude that 
	\begin{equation*}
\varphi_{r,u}''(1)=0,
	\end{equation*}
	again a contradiction.  From all these contradictions we conclude that $h'(u)$ and $g'(u)$ are linearly independent for each $u\in\mathcal{N}_r^+$. Moreover a careful look to the previous calculations shows that $g'(u)=0$ is impossible, since in that case we would have $c=0$, which gives a contradiction in all cases. Therefore $g'(u)\neq0$ and $\mathcal{N}_r^+$  is a $C^1$ manifold with co-dimension $2$ in $H^1(\mathbb{R}^3)$.
\end{proof}
Now we prove that $\mathcal{N}_r^+$ and $\mathcal{N}_r^-$ are natural constraints for the energy functional $E$.
\begin{lemma}\label{naturalcons}Assume that there exists $u\in \mathcal{N}_r^+\cup\mathcal{N}_r^-$ and $\mu,\nu\in\mathbb{R}$ such that
	\begin{equation*}
	E'(u)=\mu h'(u)+\nu g'(u)
	\end{equation*} 
	where $h$ and $g$ are given by \eqref{hg}. 	Then $\nu=0$.
\end{lemma}
\begin{proof}  Indeed, applying Proposition \ref{pohozaev} to the equation $E'(u)-\mu h'(u)-\nu g'(u)=0$  we conclude 
that
	\begin{equation*}
	\frac{3}{2}(E'(u)u-\mu h'(u)u-\nu g'(u)u)-P(u)=0.
	\end{equation*}
	Simple calculations shows that
	\begin{equation*}
	\frac{3}{2}(E'(u)u-\mu h'(u)u-\nu g'(u)u)-P(u)=g(u)-\nu \varphi_{r,u}''(1),
	\end{equation*}
	which implies that
	%\begin{equation*}
	$\nu \varphi_{r,u}''(1)=0,$
	%\end{equation*}
	and hence $\nu=0$.
\end{proof}
Lemma  \ref{manifold} and Lemma \ref{naturalcons} are summarised in the next
\begin{theorem}\label{naturalconstraints} 
%Assume that $\mathcal{N}_r^+$ ($\mathcal{N}_r^-$) is non-empty. Then $\mathcal{N}_r^+$ $(\mathcal{N}_r^-)$ is a $C^1$ manifold in $H^1(\mathbb{R}^3)$
%of co-dimension $2$. Moreover $\mathcal{N}_r^+$ ($\mathcal{N}_r^-$) is a natural constraint for the functional $E$.
Whenever nonempty, $\mathcal{N}_r^+$ and $\mathcal{N}_r^-$ are $C^1$ manifold in $H^1(\mathbb{R}^3)$
of co-dimension $2$ and natural constraints for $E.$
\end{theorem}

\medskip

The next step is then to see  for which values of $q,\lambda,p,r$  the sets $\mathcal N_{r}, \mathcal{N}_r^+,\mathcal{N}_r^-$ are non-empty. 
As a consequence of this study, we 
will be able to recover some results known in the literature by our unified approach.

\section{Structure of $\mathcal N_{r}, \mathcal N_{r}^{+}$ and $\mathcal N_{r}^{-}$}
\label{sec:nehari}

The structure of $\mathcal N_{r}, \mathcal N_{r}^{+}$ and $\mathcal N_{r}^{-}$
strongly depends on the values of $p$ and indeed different approaches are needed. 
The particular value $p=3$ is treated separately.
% It is convenient  now to consider three cases since different approaches are needed:
%\begin{itemize}
%\item $p\in(2,8/3] \cup [10/3,6)$, \smallskip
%\item  $p\in(8/3,10/3)\setminus\{3\}$, \smallskip
%\item $p=3$.
%\end{itemize}
%%, that are already know in the literature, in a unified way. 
%%\subsection{Nehari Manifolds $p\neq 3$}

\subsection{The case $p\neq3$ and Proof of Theorem \ref{Nehari} }
In is convenient  to consider the cases $p\in (2,8/3] \cup [10/3, 6]$  and $p\in(8/3,10/3)\setminus\{3\}$. 
\subsubsection{The case $p\in(2,8/3] \cup [10/3,6)$}\label{subsec:p26}

In this case we can give a simple description of $\mathcal N_{r}$. We prefer to state separately the limit cases
$p=8/3$ and $p=10/3$.
% We start with the cases $p\in(2,{8}/{3})$ and $p\in ({10}/{3},6)$:
\begin{theorem}\label{i-ii} Let $r>0$. Then
	\begin{enumerate}
		\item[i)] if $p\in (2,{8}/{3})$, then $\mathcal{N}_r=\mathcal{N}_r^+\neq\emptyset$; \medskip
		\item[ii)] if $p\in ({10}/{3},6)$, then $\mathcal{N}_r=\mathcal{N}_r^-\neq\emptyset$.
	\end{enumerate}
\end{theorem}
\begin{proof} The proof of $i)$ is a direct consequence of Proposition \ref{fiberingmaps} item $I)$ since for each $u\in S_1$ we have that $r^{{1}/{2}}u^{t_r^+(u)}\in \mathcal{N}_r^+$. Similarly, the proof of $ii)$ is a direct consequence of Proposition \ref{fiberingmaps} item $V)$, since for each $u\in S_1$ we have that $r^{{1}/{2}}u^{t_r^-(u)}\in \mathcal{N}_r^-$.
\end{proof}

\begin{theorem}\label{iii} 
Let $r>0$. If $p={8}/{3}$, then $\mathcal{N}_r=\mathcal{N}_r^+\neq\emptyset$.
\end{theorem}
\begin{proof} In fact, from Proposition \ref{fiberingmaps} item $II)$ it is sufficiently to prove that there exists $u\in S_1$ such that $\displaystyle\frac{r^2}{4}\int \phi_uu^2-\frac{r^{{p}/{2}}}{p}\int |u|^p<0$. If $\{u_n\}\subset S_1$ is the sequence given by Lemma \ref{catto}, then
	\begin{equation*}
	\lim_{n\to \infty}\left(\frac{r^2}{4}\int q\phi_{u_n}u_n^2-\frac{r^{{p}/{2}}}{p}\lambda\int |u_n|^p\right)\le\lim_{n\to \infty} \left(\frac{C_3}{n^{{2}/{3}}}\frac{r^2}{4}q-C_1\frac{r^{{p}/{2}}}{p}\lambda\right)=-C_1\frac{r^{{p}/{2}}}{p}\lambda.
	\end{equation*}
	Therefore for $n$ sufficiently large, we have that $r^{{1}/{2}}u_n^{t_r^+(u_n)}\in \mathcal{N}_r^+$.
	\end{proof}

\begin{theorem}\label{iv} Let $r>0$. If $p={10}/{3}$. %There exist a positive constant $K$ such that 
Then $\mathcal{N}_r\neq\emptyset$ if and only if 
		\begin{equation*}
	\frac{5}{3K_{\emph{GN}}}\frac{1}{\lambda}<r^{{2}/{3}}
	\end{equation*}
	(as usual $K_{\emph{GN}}$ is the Gagliardo-Nirenberg constant as in \eqref{GN}).
	Moreover in this case it is %$\mathcal{N}_r\neq\emptyset$, then
	 $\mathcal{N}_r=\mathcal{N}_r^-$.
\end{theorem}
\begin{proof}
By Proposition \ref{fiberingmaps} item $IV)$ it is sufficient to estimate, for $u\in S_{1}$,  the quantity
 $\displaystyle{\frac{r}{2}\int |\nabla u|^2-\frac{r^{{p}/{2}}}{p}\int |u|^p}$.
%  for $u\in  S_1$. %From Lemma \ref{ineqcatto} 
By the Gagliardo-Nirenberg inequality \eqref{GN} we have that
	\begin{equation*}\label{imme}
	\int|u|^p\le K_{\textrm{GN}}\int |\nabla u|^2,\quad \forall u\in S_1,\qquad \left(\text{where} 
	\  K_{\textrm{GN}}= \sup_{u\in S_{1}} \frac{\displaystyle\int|u|^p}{\displaystyle\int |\nabla u|^2}\right).
	\end{equation*}
%	where $K$ is the supremum of $\frac{\displaystyle\int|u|^p}{\displaystyle\int |\nabla u|^2}$ over $S_1$.
	It follows that
	\begin{eqnarray*}
	\frac{r}{2}\int |\nabla u|^2-\frac{r^{{p}/{2}}}{p}\lambda\int |u|^p&\ge& 	
	\frac{r}{2}\int |\nabla u|^2-K_{\textrm{GN}}\frac{r^{{p}/{2}}}{p}\lambda\int |\nabla u|^2\\
	&=&r\int |\nabla u|^2\left(\frac{1}{2}-\frac{3K_{\textrm{GN}}}{10}r^{{2}/{3}}\lambda\right).
	\end{eqnarray*}
	By definition of $K_{\textrm{GN}}$, there exist $u\in S_1$ with 
	 $\displaystyle\frac{r}{2}\int |\nabla u|^2-\frac{r^{{p}/{2}}}{p}\lambda\int |u|^p<0$ if, and only if, 
	\begin{equation*}
	\frac{5}{3K_{\textrm{GN}}}\frac{1}{\lambda}<r^{{2}/{3}},
	\end{equation*}
	in which case $r^{{1}/{2}}u^{t_r^-(u)}\in \mathcal{N}_r^-$.
\end{proof}

\medskip

\subsubsection{The case $p\in(8/3,10/3)\setminus\{3\}$}\label{subsec:p83103}
 In this case the description of $\mathcal N_{r}$
is more involved.
% and depend on the quotient $R_{p}$ defined in \eqref{NRQ}.
%The case $p\in ({8}/{3},{10}/{3})\setminus\{3\}$ will be analyzed by using 
 We use the ideas introduced by Il'yasov 
\cite{Ily}: for $r>0$ and $u\in S_1$, consider the system (recall the definitions in Subsection \ref{subsec:fibration})
$$\varphi_{r,u}(s)=\varphi_{r,u}'(s)=0.$$
 Since $p\in(8/3, 10/3) \setminus\{3\}$ we  can solve it with respect to the variables $s$ and $r$ to obtain a unique solution,
 denoted hereafter with $(s_0(u),r_0(u))$, given by
\begin{equation*}
s_0(u)=\left(\frac{p}{3p-8}\frac{1}{r^{(p-2)/2}}\frac{\displaystyle\int |\nabla u|^2}{\lambda\displaystyle\int |u|^p}\right)^{\frac{2}{3p-10}},
\end{equation*}
and
\begin{equation*}
r_0(u)=\left(\frac{2(10-3p)}{3p-8}\right)^{\frac{3p-10}{4(p-3)}}\left(\frac{p}{3p-8}\right)^{\frac{1}{2(p-3)}}R_p(u),
\end{equation*}
where $R_p$ is defined in Section \ref{sec:A}. %and studied in Appendix A, see \eqref{NRQ}. 
The following proposition is just a consequence of the definitions
and makes clear that $p=3$ is a threshold.
\begin{proposition}\label{r0ine} Assume that $p\in ({8}/{3},{10}/{3})\setminus\{3\}$.Then for each $u\in S_1$, there exist a unique pair $(s_0(u),r_0(u))$ such that $\varphi_{r_0(u),u}(s_0(u))=\varphi'_{r_0(u),u}(s_0(u))=0$. Moreover
\begin{enumerate}
	\item If $p\in ({8}/{3},3)$ and $r<r_0(u)$, then $\varphi_{r,u}(s_0(u))<0$ and $\varphi_{r,u}'(s_0(u))=0$, while if $r>r_0(u)$, then $\varphi_{r,u}(s_0(u))>0$ and $\varphi_{r,u}'(s_0(u))=0$.  \smallskip
	\item If $p\in (3,{10}/{3})$ and $r<r_0(u)$, then $\varphi_{r,u}(s_0(u))>0$ and $\varphi_{r,u}'(s_0(u))=0$, while if $r>r_0(u)$, then $\varphi_{r,u}(s_0(u))<0$ and $\varphi_{r,u}'(s_0(u))=0$.
\end{enumerate}
\end{proposition}
Similarly, for $r>0$ and $u\in S_1$ we consider the system $$\varphi_{r,u}'(s)=\varphi_{r,u}''(s)=0.$$ 
Again, since $p\neq 3$ (and $p\neq10/3$), we  can solve it with respect to the variables $s$ and $r$ to obtain a unique solution,
hereafter denoted with $(s(u),r(u))$, given by 
\begin{equation*}
s(u)=\left(\frac{4p}{3(p-2)(3p-8)}\frac{1}{r^{(p-2)/2}}\frac{\displaystyle\int |\nabla u|^2}{\lambda\displaystyle\int |u|^p}\right)^{\frac{2}{3p-10}},
\end{equation*}
and
\begin{equation*}
r(u)=\left(4\frac{10-3p}{3p-8}\right)^{\frac{3p-10}{4(p-3)}}\left(\frac{4p}{3(p-2)(3p-8)}\right)^{\frac{1}{2(p-3)}}R_p(u).
\end{equation*}
Similarly to Proposition \ref{r0ine} we have:
\begin{proposition}\label{rineq} Assume that $p\in ({8}/{3},{10}/{3})\setminus\{3\}$. Then for each $u\in S_1$, there exist a unique pair $(s(u),r(u))$ such that $\varphi'_{r(u),u}(s(u))=\varphi''_{r(u),u}(s(u))=0$. Moreover
	\begin{enumerate}
		\item If $p\in ({8}/{3},3)$ and $r<r(u)$, then $\varphi_{r,u}'(s(u))<0$ and $\varphi_{r,u}''(s(u))=0$, while if $r>r(u)$, then $\varphi_{r,u}'(s(u))>0$ and $\varphi_{r,u}''(s(u))=0$. \smallskip
		\item If $p\in (3,{10}/{3})$ and $r<r(u)$, then $\varphi_{r,u}'(s(u))>0$ and $\varphi_{r,u}''(s(u))=0$, while if $r>r(u)$, then $\varphi_{r,u}'(s(u))<0$ and $\varphi_{r,u}''(s(u))=0$.
	\end{enumerate}
	\end{proposition}

		Furthermore 

\begin{proposition}\label{extremalfunctions} For each $u\in S_1$ we have that:
	\begin{enumerate}
		\item[i)] if $p\in ({8}/{3},3)$, then $r_0(u)<r(u)$; \smallskip
				\item[ii)] if $p\in (3,{10}/{3})$, then $r_0(u)>r(u)$.
	\end{enumerate}
Moreover
\begin{enumerate}
	\item[iii)] if $p\in ({8}/{3},3)$, then the functions $S_1\ni u\mapsto r_0(u),r(u)$ are unbounded from above;\smallskip
	\item[iv)] if $p\in (3,{10}/{3})$, then the functions $S_1\ni u\mapsto r_0(u),r(u)$ are bounded away from zero.
\end{enumerate}
\end{proposition}
\begin{proof} The proofs of $i)$ and $ii)$ are straightforward and the proofs of $iii)$ and $iv)$ are consequence of Proposition \ref{NQRP}.
	\end{proof}
To treat the case $p\in (3,{10}/{3})$ we need also the numbers 
\begin{equation}\label{eq:rstar}
r_0^*:=\inf_{u\in S_1}r_0(u)\ \ \ \mbox{and}\ \ \ r^*:=\inf_{u\in S_1}r(u).
\end{equation}
Then the description of $\mathcal N_{r}, \mathcal N_{r}^{+},  \mathcal N_{r}^{-}$
is given.

\begin{theorem}\label{nehari83103} There hold:
	\begin{enumerate}
		\item[i)] suppose that $p\in ({8}/{3},3)$. Then for each $r>0$ there exists $u\in S_1$ such that $\inf_{t>0}\varphi_{r,u}(t)<0$. Moreover $\mathcal{N}_r^+$ and $\mathcal{N}_r^-$ are non-empty. \medskip
		\item[ii)] Suppose that $p\in (3,{10}/{3})$. If $r<r^*$, then $\mathcal{N}_r=\emptyset$, while if $r>r^*$, then $\mathcal{N}_r^+$ and $\mathcal{N}_r^-$ are non-empty. Moreover if $r<r_0^*$, then $\inf_{t>0}\varphi_{r,u}(t)\ge0$ for each $u\in S_1$, while if $r>r_0^*$, then there exists $u\in S_1$ such that $\inf_{t>0}\varphi_{r,u}(t)<0$.
	\end{enumerate}
\end{theorem}
\begin{proof} $i)$ Fix $r>0$ and assume on the contrary that for each $u\in S_1$ we have that $\inf_{t>0}\varphi_{r,u}(t)\ge 0$. In particular, it follows that $\varphi_{r,u}(t_r^+(u))\ge 0$ and therefore, from Proposition \ref{r0ine} we conclude that $r_0(u)\le r$ for all $u\in S_1$. This contradicts 
Proposition \ref{extremalfunctions} (iii) and hence 
there exists $u\in S_1$ such that $\inf_{t>0}\varphi_{r,u}(t)<0$. To conclude, note from Proposition \ref{fiberingmaps} item $III)$ that if  $u\in S_1$ satisfies $\inf_{t>0}\varphi_{r,u}(t)<0$, then $r^{{1}/{2}}u^{t_r^-(u)}\in\mathcal{N}_r^-$ and $r^{{1}/{2}}u^{t_r^+(u)}\in\mathcal{N}_r^+$.

\medskip

$ii)$ Fix $r<r^*$ and suppose on the contrary that $\mathcal{N}_r\neq\emptyset$. Take $u\in \mathcal{N}_r$ and observe from Proposition \ref{fiberingmaps} item $III)$ that there exists $\overline{t}>0$ such that $\varphi_{r,u}'(\bar{t})\le 0$ and $\varphi_{r,u}''(\bar{t})= 0$. From Proposition \ref{rineq} we  conclude that $r\ge r(u)\ge r^*$ which is clearly a contradiction and therefore $\mathcal{N}_r=\emptyset$.

 Now fix $r>r^*$ and assume on the contrary that $\mathcal{N}_r^+=\emptyset$, which implies from Proposition \ref{fiberingmaps} item $III)$ that $\mathcal{N}_r^-=\emptyset$ (and vice-versa). From the same proposition, we  conclude that for each $u\in S_1$, when $\varphi_{r,u}''(t)= 0$ then $\varphi_{r,u}'(t)\ge0$. It follows from Proposition \ref{rineq}  that  $r<r(u)$ for all $u\in S_1$, again a contradiction and hence $\mathcal{N}_r^+$ and $\mathcal{N}_r^-$ are non-empty. 
 
 By using the function $S_1\ni u\mapsto r_0(u)$ instead of $S_1\ni u\mapsto r(u)$, the rest of the proof is similar.
\end{proof}
%\subsubsection{Proof of Theorem \ref{Nehari} }

\medskip

Now we can give the proof of Theorem \ref{Nehari}. 

\noindent Indeed $i)$ follows by Theorem \ref{i-ii} and Proposition \ref{fiberingmaps} item $I)$.
$ii)$ Follows by Theorem \ref{i-ii} and Proposition \ref{fiberingmaps} item $V)$.
$iii)$ Follows by Theorem \ref{iii}.
$iv)$ Follows by Theorem \ref{iv}.
$v)$ and  $vi)$ follow by Theorem \ref{nehari83103}.
%It follows from Theorems \ref{i-ii}, \ref{iii}, \ref{iv}, \ref{nehari83103} and Proposition \ref{fiberingmaps}. 
%\end{proof}

\subsection{The case $p=3$ and Proof of Theorem \ref{Nehari1}}\label{subsec:p=3}
In this case the system $$\varphi_{r,u}(t)=\varphi_{r,u}'(t)=0$$ has no solution with respect to the variables $t,r$.
 Therefore, instead of the variable $r$, we will solve the system with respect to the variable $\lambda$ and analyze the dependence of the solutions with respect to $q$. It will be 
 clear from the calculations that, at least topologically speaking, there are no changes in the fibering maps with respect to $r$, hence, to reflect the dependence on $q,\lambda$, 
 we change the notation here; so for example we will write
 $\mathcal{N}_{q,\lambda}, \varphi_{q,\lambda, u}, \ldots$ instead of $\mathcal{N}_r, \varphi_{r,u}, \ldots$ 
 we used up to now.

\medskip

Consider then the system of equations $\varphi_{q,\lambda,u}(t)=\varphi_{q,\lambda,u}'(t)=0$. 
We solve this system with respect to the variables $t,\lambda$ to find a unique solution given by
\begin{equation*}
t_{0,q}(u)=\left(\frac{3}{r^{{1}/{2}}}\frac{\displaystyle\int |\nabla u|^2}{\lambda\displaystyle\int |u|^3}\right)^{-2},
\end{equation*}
and
\begin{equation*}
\lambda_{0,q}(u)=\left(\frac{9}{2}\right)^{\frac{1}{2}}q^{\frac{1}{2}}\frac{\left(\displaystyle\int|\nabla u|^2\displaystyle\int\phi_uu^2\right)^{\frac{1}{2}}}{\displaystyle\int|u|^3}.
\end{equation*}
Similarly, we consider the system $\varphi_{q,\lambda,u}'(t)=\varphi_{q,\lambda,u}''(t)=0$ and solve it with respect to the variables $t$ and $\lambda$ to obtain a unique solution given by 
\begin{equation*}
t_q(u)=\left(\frac{4}{r^{{1}/{2}}}\frac{\displaystyle\int |\nabla u|^2}{\lambda\displaystyle\int |u|^3}\right)^{-2},
\end{equation*}
and
\begin{equation*}
\lambda_q(u)=2q^{\frac{1}{2}}\frac{\left(\displaystyle\int|\nabla u|^2\displaystyle\int\phi_uu^2\right)^{\frac{1}{2}}}{\displaystyle\int|u|^3}.
\end{equation*}
As an application of Lemma \ref{ineqcatto} we have:
\begin{proposition}\label{extremalp=3} For each $r,q>0$, the functions $S_1\ni u\mapsto \lambda_{0,q}(u),\lambda_{q}(u)$ are bounded away from zero.
%\textcolor{red}{(separar em 2 partes pois a $r$ se define quando $p\neq3$}.
Moreover $\lambda_q(u)<\lambda_{0,q}(u)$ for all $u\in S_1$.
\end{proposition}
For each $r,q>0$ define
\begin{equation*}
\lambda_{0,q}^*:=\inf_{u\in S_1}\lambda_{0,q}(u)\ \ \ \mbox{and}\ \ \ \lambda_q^*:=\inf_{u\in S_1}\lambda_q(u).
\end{equation*}

\medskip

Then the proof of Theorem \ref{Nehari1} can be finished.
%	\subsubsection{Proof of Theorem \ref{Nehari1}} 
	 Indeed it is similar to the proof of Theorem \ref{nehari83103} (we use Proposition \ref{extremalp=3} 
	instead of Proposition \ref{extremalfunctions}).

\section{On the sub-additive property for $p\in(2,10/3)$} \label{sec:subad}
For each $p\in (2,10/3)$ %with $p\neq 3$,
define 
\begin{equation*}
I_r:=I_{r,q,\lambda}=\inf\{E(u):u\in \mathcal{N}_r^+\cup \mathcal{N}_r^0\}.
\end{equation*}
Since $E$ is bounded from below on $S_r$
(see e.g. \cite[Lemma 3.1]{bellasici2}) and $\mathcal N_{r}\subset S_{r}$, we conclude 
from Theorem \ref{Nehari} that $I_r$ is well defined,
that is  $I_r>-\infty$. 

In this Section we show how our method can be used to prove the sub-additive condition
for $I_r$, namely
\begin{equation}\label{eq:SAII}
I_r < I_s +I_{r-s},\quad 0<s<r.
\end{equation}
Again it is convenient to study separately the case $p=3$.

\subsection{The case $p\in(2,10/3)\setminus\{3\}$}
%\textcolor{red}{We already know that $R_p$ is bounded away 
%from $0$ when $p\in (3,10/3)$. In order to study system \eqref{sys1} we need to 
%prove this result also for $p\in(2,3)$. }

\medskip
%
%\textcolor{blue}{In particular by \eqref{ineqcattooo} and 
%	the inequality in Theorem \ref{cattoinep<3} it follows that there exists a constant $C'>0$
%	such that
%	\begin{equation}\label{eq:nova?}
%	\int \phi_{u}u^{2} \leq C' \left( \int u^{2} \right)^{3/2}  \left(\int |\nabla u|^{2} \right) ^{1/2}.
%	\end{equation}
%}

We recall that, given $u\in S_1$, by definition  $(\widetilde r(u), \widetilde t(u))$ is the unique solution of
\begin{equation*}
\left\{
\begin{aligned}
rt^2\int |\nabla u|^2+\frac{r^2t}{4}q\int \phi_uu^2-\frac{3(p-2)}{2p}r^{{p}/{2}}t^{\frac{3(p-2)}{2}}\lambda\int |u|^p=0, \\
\frac{q}{2}r^2t\int \phi_uu^2-\frac{p-2}{p}r^{{p}/{2}}t^{\frac{3(p-2)}{2}}\lambda\int |u|^p=0, 
\end{aligned}
\right.
\end{equation*}
see \eqref{eq:no8/3} and \eqref{eq:8/3} for the explicit value of the solutions.

\begin{proposition}\label{tildepro} For each $p\in(2,10/3)\setminus\{3\}$, the functional $S_1\ni u\mapsto \widetilde{r}(u)$ is bounded away from $0$. Moreover,
	\begin{enumerate}
		\item[i)] if $p\in (3,10/3)$, then $r(u)<\widetilde{r}(u)<r_0(u)$, for all $u\in S_1$; \medskip
		\item[ii)] if $p\in (8/3,3)$, then $\widetilde{r}(u)<r_0(u)<r(u)$, for all $u\in S_1$.
	\end{enumerate} 
\end{proposition}
\begin{proof} That $S_1\ni u\mapsto \widetilde{r}(u)$ is bounded away from $0$, for all $p\in(2,10/3)\setminus\{3\}$, 
	follows from Proposition \ref{NQRP} and Theorem \ref{th:boundedbelow}.
	 The proofs of $i)$ and $ii)$ are straightforward.
\end{proof}
%\begin{proposition} Suppose that $p\in(2,10/3)\setminus\{3\}$. Let $\{u_n\}\subset  S_1$, $\{ t_n \}, \{r_n\}\subset(0,+\infty)$ be sequences 
%satisfying: there exist positive constants $c,C$ such that $c\le r_n\le C$,
%	\begin{equation*}
%	\int |\nabla u_n|^2\ge c, \ \ \ \int \phi_{u_n}u_n^2\ge c \ \ \ \mbox{and} \ \ \ \int|u_n|^p\ge c,
%	\end{equation*}
%	and 
%	\begin{equation*}
%	\left\{
%	\begin{aligned}
%	r_nt_n^2\int |\nabla u_n|^2+\frac{r_n^2t_n}{4}q\int \phi_{u_n}u_n^2-\frac{3(p-2)}{2p}r_n^{{p}/{2}}t_n^{\frac{3(p-2)}{2}}\lambda\int |u_n|^p=0, \\
%	\frac{q}{2}r_n^2t_n\int \phi_{u_n}u_n^2-\frac{p-2}{p}r_n^{{p}/{2}}t_n^{\frac{3(p-2)}{2}}\lambda\int |u_n|^p=o_n(1).
%	\end{aligned}
%	\right.
%	\end{equation*}
%	Then $r_n=\widetilde{r}(u_n)+o_n(1)$ as $n\to \infty$.
%\end{proposition}
%\begin{proof} First we deal with the case $p\neq 8/3$. From the hypothesis it follows that there exist positive constants $d,D$ such that $d\le t_n\le D$ for all $n$. 
%\textcolor{red}{nao deve ser $\|\nabla u_n\|_2\le C ?$}	
%	Therefore, from the second equation, we conclude that
%	\begin{equation*}
%t_n=\left(\frac{p}{2(p-2)} r_n^{\frac{4-p}{2}}\frac{q}{\lambda}\frac{\displaystyle\int \phi_{u_n}u_n^2}{\displaystyle\int 
%	|u_n|^p}\right)^{\frac{2}{3p-8}}+o_n(1).
%	\end{equation*}
%Now we plug $t_n$ on the first equation to obtain that $r_n=\widetilde{r}(u_n)+o(1)$ as $n\to \infty$.
%\end{proof}

As was already observed (see e.g. \cite{bellasici1}), in order to prove the strict sub-additive condition \eqref{eq:SAII}, it is sufficiently to show that 
$I_r/r$ is decreasing in $r$. 
Our strategy to prove that $I_r/r$ is decreasing in $r$ will be the following: we will construct paths 
that cross the Nehari manifolds when $r$ varies and show that the energy restricted to these paths, divided by $r$, is 
decreasing. Then we will show that the function $I_r/r$ will inherit this property for some specific values of $r$.

Fix $u\in S_1$ and 
\begin{enumerate}
	\item[i)]  if $p\in (2,8/3)$, define $f(r):=\varphi_{r,u}(t_r^+(u))$ for all $r\in (0,\infty)$; \medskip
	\item[ii)] if $p=8/3$ and $\displaystyle\frac{r^2}{4}\int \phi_uu^2-\frac{3}{8}r^{{4}/{3}}\lambda\int |u|^{{8}/{3}}<0$, 
	define $f(r):=\varphi_{r,u}(t_r^+(u))$ for all $r\in(0,r(u))$ where, in this case, $r(u)$ is by definition the unique $r>0$ for which 
	$$\frac{r^2}{4}\int \phi_uu^2-\frac{3}{8}r^{{4}/{3}}\lambda\int |u|^{{8}/{3}}=0;$$ \medskip
	\item[iii)] if $p\in (8/3,3)$, define $f(r):=\varphi_{r,u}(t_r^+(u))$ for all $r\in (0,r(u))$; \medskip
	\item[iv)] if $p\in (3,10/3)$, define $f(r)=\varphi_{r,u}(t_r^+(u))$ for all $r\in (r(u),\infty)$. \medskip
%	\item[v)] In all cases let $g(r)=f(r)/r$.
\end{enumerate}
Define also
$$g(r):=\frac{f(r)}{r}.$$
% \begin{remark}
Clearly  $f$, and consequently $g$, depends on $u\in S_1$.
% \end{remark}
\begin{proposition}\label{decreasinir} Let $u\in S_1$.
	\begin{enumerate}
		\item[i)] If $p\in (2,8/3)$, then the function $(0,\infty)\ni r\mapsto g(r)$ is decreasing for all $r\in (0,\widetilde{r}(u))$ and  increasing for all $r\in (\widetilde{r}(u),r(u))$. \medskip
		\item[ii)]  If $p=8/3$, then the function $(0,\infty)\ni r\mapsto g(r)$ is decreasing for all $r\in (0,\widetilde{r}(u))$ and increasing for $r\in (\widetilde{r}(u),r(u))$.\medskip
		\item[iii)] If $p\in (8/3,3)$, then the function $(0,r(u))\ni r\mapsto g(r)$ is decreasing for all $r\in (0,\widetilde{r}(u))$.\medskip
		\item[iv)] If $p\in (3,10/3)$, then the function $(r(u),\infty)\ni r\mapsto g(r)$ is decreasing for all $r\in (\widetilde{r}(u),\infty)$.
	\end{enumerate}
\end{proposition}
\begin{proof} Indeed, from the definition of $f(r)$, it follows from Lemma \ref{diffet+} that $g$ is $C^1$ and 
	\begin{eqnarray*}
	g'(r)&=&r\varphi'_{r,u}(t_r^+(u))+\frac{q}{2}t_r^+(u)\int \phi_uu^2-\frac{p-2}{p}r^{{p}/{2}-2}t_r^+(u)^{\frac{3p}{2}-3}\lambda\int |u|^p.
	%&=& \frac{q}{2}t_r^+(u)\int \phi_uu^2-\frac{p-2}{p}r^{{p}/{2}-2}t_r^+(u)^{\frac{3p}{2}-3}\lambda\int |u|^p.
	\end{eqnarray*}
	For simplicity denote $t_r=t_r^+(u)$. It follows that $g'(r)=0$ if, and only if
	\begin{equation}\label{sys1}
	\left\{
	\begin{aligned}
	rt_r\int |\nabla u|^2+\frac{r^2}{4}q\int \phi_uu^2-\frac{3(p-2)}{2p}r^{{p}/{2}}t_r^{\frac{3p}{2}-4}\lambda\int |u|^p=0, \\
	\frac{q}{2}t_r\int \phi_uu^2-\frac{p-2}{p}r^{{p}/{2}-2}t_r^{\frac{3p}{2}-3}\lambda\int |u|^p=0, 
	\end{aligned}
	\right.
	\end{equation}
	which is equivalent to system \eqref{decre}. Fixed $r>0$, define 
	\begin{equation*}
	h(t):=\frac{q}{2}\int \phi_uu^2-\frac{p-2}{p}r^{\frac{p}{2}-2}t^{\frac{3p}{2}-4}\lambda\int |u|^p.
	\end{equation*} We consider two cases:
	
	\vskip.3cm
	{\bf Case 1: $p=8/3$}.
	\vskip.3cm
	Observe that the first equation of \eqref{decre} has a unique solution $t$. By plugging this solution on the left hand side of the second equation, 
	which is exactly $th(t)$, the proof of $ii)$ is complete.
	
	\vskip.3cm
	{\bf Case 2: $p\in(2,10/3)\setminus\{8/3,3\}$}.
	\vskip.3cm
	Note that the second equation of \eqref{decre} has a unique solution $t$. By plugging this solution on the left hand side of the first equation, which is exactly $\varphi'_{r,u}(t)$, we conclude, by using Proposition \ref{fiberingmaps}, the following:
	\begin{itemize}
		\item[1)] if $p\in (2,8/3)$ and $r\in(0,\widetilde{r}(u))$, then $\varphi'_{r,u}(t)>0$, while for $r\in(\widetilde{r}(u),r(u))$ we have that $\varphi'_{r,u}(t)<0$;\medskip
		\item[2)] if $p\in (8/3,3)$ and $r\in(0,\widetilde{r}(u))$, then $\varphi'_{r,u}(t)<0$; \medskip		
		\item[3)] if $p\in (3,10/3)$ and $r\in (\widetilde{r}(u),\infty)$, then $\varphi'_{r,u}(t)<0$.
	\end{itemize}
	
Now we can prove $i), iii)$ and $iv)$. 

$i)$ If $r\in(0,\widetilde{r}(u))$, then from item 1), we conclude that $t>t_r$ and hence $h(t_r)<h(t)=0$, while if 
$r\in(\widetilde{r}(u),r(u))$, then $t<t_r$ and hence $h(t_r)>h(t)=0$. 

\medskip

$iii)$ If $r\in(0,\widetilde{r}(u))$, then from item 2), we conclude that 
$t<t_r$ and hence $h(t_r)<h(t)=0$, that is $g'(r)<0$. 

\medskip

$iv)$ If $r\in (\widetilde{r}(u),\infty)$, then from item 3), we conclude that 
$t<t_r$ and hence $h(t_r)<h(t)=0$, that is $g'(r)<0$.
\end{proof}
Let us define now
\begin{equation*}%\label{eq:M}
\mathcal{M}_r=\left\{\frac{u}{\|u\|_2}: u\in\mathcal{N}_r^+\ \mbox{and}\ E(u)<0\right\}.
\end{equation*}
\begin{lemma}\label{containe} There holds:
	\begin{enumerate}
		\item[i)] if $p\in (2,3)$ and $0<r_1<r_2<r^*$, then $\mathcal{M}_{r_1}=\mathcal{M}_{r_2}=S_1$; \medskip
		\item[ii)]  if $p\in (3,10/3)$ and $r^*<r_1<r_2$, then $\mathcal{M}_{r_1}\subset \mathcal{M}_{r_2}$.
	\end{enumerate} 
\end{lemma}
\begin{proof} $i)$ Fix $0<r<r^*$. Then, $\varphi_{r,u}$ satisfies Item $III-1)$ of Proposition \ref{fiberingmaps} for all $u\in S_1$ and hence $\mathcal{M}_{r_1}=\mathcal{M}_{r_2}=S_1$. 

\medskip

	$ii)$ Indeed, if $u\in \mathcal{M}_{r_1}$, then the fiber map $\varphi_{r_1,u}$ satisfies Item $III-1)$ of Proposition \ref{fiberingmaps}. From Proposition \ref{rineq} it follows that $r(u)<r_1<r_2$ and hence $\varphi_{r_2,u}$ also satisfies Item $III-1)$ of Proposition \ref{fiberingmaps}, which implies that $\mathcal{M}_{r_1}\subset \mathcal{M}_{r_2}$.
\end{proof}

\begin{lemma}\label{decreasingcomc} Suppose that $p\in(3,10/3)$ and let $r\in[a,b]$ where $r_0^*<a<b$. Then there exists a negative constant $c$ 
	such that $g'(r)<c_r$ for all $u\in  \mathcal{M}_r$ and $r\in [a,b]$.
\end{lemma}
\begin{proof} In order to prove the lemma, it is sufficiently to prove that the left hand side of the second equation of system \eqref{sys1} is bounded 
from above by $c$ for all $u\in  \mathcal{M}_r$ and $r\in [a,b]$. 

First observe from Proposition \ref{tildepro} that $\widetilde{r}(u)<r_0(u)<r$ for all 
$u\in  \mathcal{M}_r$ and all $r\in[a,b]$ and hence from Theorem \ref{decreasinir}, we conclude that $g'(r)<0$ for all $u\in  \mathcal{M}_r$ and $r\in 
[a,b]$.  Now note that $g(r)=\varphi_{r,u}(t_r^+(u))/r=\varphi_{r,su}(t_r^+(su))/r$ for all $s>0$ and therefore, by choosing $s=1/\|\nabla u\|_2$, we can 
assume that $\|\nabla u\|_2=1$ for all $u\in  \mathcal{M}_r$. 

Suppose on the contrary that there exists a sequence $\{u_n\} \subset \mathcal{M}_r$ 
with $\|\nabla u_n\|_2=1$ and corresponding sequences $t_n>0$, $r_n\in[a,b]$ such that 
	\begin{equation}\label{sys2}
	\left\{
	\begin{aligned}
	r_nt_n\int |\nabla u_n|^2+\frac{r_n^2}{4}q\int \phi_{u_n}u_n^2-\frac{3(p-2)}{2p}r_n^{{p}/{2}}t_n^{\frac{3p}{2}-4}\lambda\int |u_n|^p=0, \\
	\frac{q}{2}t_n\int \phi_{u_n}u_n^2-\frac{p-2}{p}r_n^{{p}/{2}-2}t_n^{\frac{3p}{2}-3}\lambda\int |u_n|^p=o_n(1), 
	\end{aligned}
	\right.
	\end{equation}
	From Proposition \ref{neharibounded} and Gagliardo-Nirenberg  inequality it follows that there exists positive constants $c,d$ such that $c\le t_n\le d$ and $c\le \int |u_n|^p\le d$ for all $n$. Therefore from the second equation of \eqref{sys2} we obtain that
	\begin{equation*}
	t_n=\left(\frac{p}{2(p-2)} r_n^{\frac{4-p}{2}}\frac{q}{\lambda}\frac{\displaystyle\int \phi_{u_n}u_n^2}{\displaystyle\int |u_n|^p}\right)^{\frac{2}{3p-8}}+o_{n}(1).
	\end{equation*}
	By plugging $t_n$ in the first equation of \eqref{sys2} we conclude that $r_n=\widetilde{r}(u_n)+o_{n}(1)$ and hence $r_n=cr_0(u_n)+o_n(1)<cr_n+o_{n}(1)$ where $c\in(0,1)$, which is a contradiction. Then there exists a negative constant $c_r$ such that $g'(r)<c_r$ for all $u\in  \mathcal{M}_r$.
\end{proof}
Since    we do not have  a priori estimates like in  Proposition \ref{neharibounded} for the case $p\in (2,3)$, a similar version of Lemma \ref{decreasingcomc} for that case is not so immediate, however, if we control the term $\int|u|^p$, then we can prove the following:
\begin{lemma}\label{decreasingcomc1} Suppose that $p\in(2,3)$ and let $r\in[a,b]$ where $0<a<b<\inf_{u\in S_1}\widetilde{r}(u)$. Fix $d>0$, then there exists a negative constant $c$ such that $g'(r)<c$ for all $u\in  \mathcal{M}_r$ satisfying $\int|u^{t_r^+(u)}|^p\ge c$ and all $r\in[a,b]$.
\end{lemma}
\begin{proof} In order to prove the lemma, it is sufficiently to prove that the left hand side of the second equation of system \eqref{sys1} is bounded from above by $c$ for all $u\in  \mathcal{M}_r$ satisfying $\int|u^{t_r^+(u)}|^p\ge d$ and all $r\in[a,b]$. From Theorem \ref{decreasinir}, we have that $g'(r)<0$ for all $u\in  \mathcal{M}_r$. Now note that $g(r)=\varphi_{r,u}(t_r^+(u))/r=\varphi_{r,su}(t_r^+(su))/r$ for all $s>0$ and therefore, by choosing $s=1/\|\nabla u\|_2$, we can assume that $\|\nabla u\|_2=1$ for all $u\in  \mathcal{M}_r$ satisfying $\int|u^{t_r^+(u)}|^p\ge d$. Suppose on the contrary that there exists a sequence $\{u_n \}\subset \mathcal{M}_r$ satisfying $\|\nabla u_n\|_2=1$ and $\int|u_n^{t_r^+(u_n)}|^p\ge d$ and corresponding sequences $t_n>0$, $r_n\in[a,b]$ such that 
	\begin{equation*}
	\left\{
	\begin{aligned}
	r_nt_n\int |\nabla u_n|^2+\frac{r_n^2}{4}q\int \phi_{u_n}u_n^2-\frac{3(p-2)}{2p}r_n^{{p}/{2}}t_n^{\frac{3p}{2}-4}\lambda\int |u_n|^p=0, \\
	\frac{q}{2}t_n\int \phi_{u_n}u_n^2-\frac{p-2}{p}r_n^{{p}/{2}-2}t_n^{\frac{3p}{2}-3}\lambda\int |u_n|^p=o_n(1).
	\end{aligned}
	\right.
	\end{equation*}
	Arguing as in the proof of Lemma \ref{decreasingcomc} we conclude that 
	$$r_n=\widetilde{r}(u_n)+o_{n}(1)\ge \inf_{u\in S_1}\widetilde{r}(u)+o_{n}(1)>b+\varepsilon+o_n(1)$$
	 for some $\varepsilon$, which is a contradiction. The proof is complete.
\end{proof}

At this point we have the desired result on  $I_r/r$.
\begin{theorem}\label{gdecreasing} There holds:
	\begin{enumerate}
		\item[i)] if $p\in(2,3)$, then the function $(0,\inf_{u\in S_1},\widetilde{r}(u))\ni r \mapsto I_r/r$ is decreasing;\medskip
		\item[ii)] if $p\in(3,10/3)$, then the function $(r_0^*,\infty)\ni r \mapsto I_r/r$ is decreasing.		
	\end{enumerate}
\end{theorem} 
\begin{proof} $i)$ Fix $0<r_1<r_2< \inf_{u\in S_1},\widetilde{r}(u)<r^*$ and let $\{u_n\}\subset \mathcal{N}_{r_1}^+$ be a minimizing sequence to 
	$I_{r_1}$. Since every such sequence is non-vanishing, we can assume that $\int|u_n|^p\ge d$ for some positive constant $d$ and all $r\in[r_1,r_2]$. From Lemma \ref{containe}, Lemma \ref{decreasingcomc1} and the mean value theorem, we conclude that,
	for all $n \in \mathbb N$,
	\begin{eqnarray*}
	\frac{I_{r_2}}{r_2}&\le& \frac{\varphi_{r_2,u_n}(t_{r_2}^+(u_n))}{r_2} \\ 
	&=& \frac{\varphi_{r_1,u_n}(t_{r_1}^+(u_n))}{r_1}+g'(\theta)(r_2-r_1)\\
	&<&\frac{\varphi_{r_1,u_n}(t_{r_1}^+(u_n))}{r_1}+c(r_2-r_1)
	\end{eqnarray*}
	where $\theta\in (r_1,r_2)$. As a consequence
	\begin{equation*}
	\frac{I_{r_2}}{r_2}\le 	\frac{I_{r_1}}{r_1}+c(r_2-r_1),
	\end{equation*}
	and the proof of $i)$ is complete. \medskip
	
	$ii)$ Fix $r_0^*<r_1<r_2$ and note from Lemma \ref{containe}, Lemma \ref{decreasingcomc} and the mean value theorem that
	\begin{eqnarray*}
	\frac{I_{r_2}}{r_2}&\le& \frac{\varphi_{r_2,u}(t_{r_2}^+(u))}{r_2} \\
	&=& \frac{\varphi_{r_1,u}(t_{r_1}^+(u))}{r_1}+g'(\theta)(r_2-r_1)\\
	&<&\frac{\varphi_{r_1,u}(t_{r_1}^+(u_n))}{r_1}+c(r_2-r_1), \quad \forall u\in \mathcal{M}_{r_1},
	\end{eqnarray*}
	where $\theta\in (r_1,r_2)$. As a consequence
	\begin{equation*}
	\frac{I_{r_2}}{r_2}\le 	\frac{I_{r_1}}{r_1}+c(r_2-r_1),
	\end{equation*}
	and the proof of is complete.
\end{proof}
As an immediate consequence of Theorem \ref{gdecreasing} we have the subadditivity inequality for $I_r$.
\begin{theorem}\label{stricsub} There holds:
		\begin{enumerate}
		\item[i)] if $p\in(2,3)$, then for each $r_1,r_2\in(0,\inf_{u\in S_1},\widetilde{r}(u))$, with $r_1<r_2$, we have that $I_{r_2}<I_{r_1}+I_{r_2-r_1}$; \medskip
		\item[ii)] if $p\in(3,10/3)$, then for each $r_1,r_2\in(r_0^*,\infty)$, with $r_1<r_2$, we have that $I_{r_2}<I_{r_1}+I_{r_2-r_1}$.		
	\end{enumerate}
\end{theorem}
\begin{remark} When $p\in (2,8/3]$ we see from Theorem \ref{decreasinir} that after $\widetilde{r}(u)$ the function $g$ is increasing. This suggest 
	that the same property may hold for $I_{r}/r$ when $r$ is big and this suggests that $I_r$ may not satisfies the strict sub-additive property. 
\end{remark}
\subsection{The case $p=3$}
We assume that $\lambda>\lambda_q^*$, which implies from Theorem \ref{Nehari1} that $\mathcal{N}_r^+\neq \emptyset$ for all $r>0$. We define in this case
\begin{equation*}%\label{eq:M2}
\mathcal{M}_r=\left\{\frac{u}{\|u\|_2}: u\in\mathcal{N}_r^+\right\}.
\end{equation*} Since, as observed in Subsection \ref{subsec:p=3}, the system $\varphi_{r,u}'(t)=\varphi_{r,u}''(t)$ does not depend on $r>0$, it follows that
\begin{lemma}\label{mmm1} There holds:
	\begin{equation*}
	\mathcal{M}_r=\mathcal{M}_1,\quad \forall r>0.
	\end{equation*}
	
\end{lemma}
From Lemma \ref{mmm1} we conclude that if $u\in \mathcal{M}_1\subset S_1$, then $t_r^+(u)$ is defined for all $r>0$ and thus we can define $f(r)=\varphi_{r,u}(t_r^+(u))$.
\begin{lemma}\label{fibercarp=3} For each $r>0$ and $u\in \mathcal{M}_1$, we have that $f(r)=f(1)r^3$.
\end{lemma}
\begin{proof} Note that
	\begin{equation*}
	\frac{f(r)}{r^3}=\frac{1}{2}\left(\frac{t_r^+(u)}{r}\right)^2\int |\nabla u|^2+\frac{1}{4}\frac{t_r^+(u)}{r}\int \phi_uu^2-\frac{1}{3}\left(\frac{t_r^+(u)}{r}\right)^\frac{3}{2}\int |u|^3.
	\end{equation*}
	Since $ru^{t_r^+(u)}\in \mathcal{N}_r^+$, we also have that
		\begin{equation*}
\left(\frac{t_r^+(u)}{r}\right)^2\int |\nabla u|^2+\frac{1}{4}\frac{t_r^+(u)}{r}\int \phi_uu^2-\frac{1}{2}\left(\frac{t_r^+(u)}{r}\right)^\frac{3}{2}\int |u|^3=0.
	\end{equation*}
	Therefore 
	\begin{equation*}
	\left(\frac{f(r)}{r^3}\right)'=0,\quad \forall r>0,
	\end{equation*}
	which completes the proof.
\end{proof}
\begin{proposition}\label{mmm2} For each $r>0$, we have that $I_r=I_1r^3$.
\end{proposition}
\begin{proof} For each $u\in \mathcal{M}_1$, we have from Lemma \ref{fibercarp=3} that
	\begin{equation*}
	\frac{\varphi_{r,u}(t_r^+(u))}{r^3}=\varphi_{1,u}(t_1^+(u)).
	\end{equation*}
	Therefore
		\begin{equation*}
	\frac{I_r}{r^3}= \inf_{u\in \mathcal{M}_1}\left\{\frac{\varphi_{r,u}(t_r^+(u))}{r^3}\right\}=\inf_{u\in \mathcal{M}_1}\varphi_{1,u}(t_1^+(u))=I_1
	\end{equation*}
	and the proof is completed.
\end{proof}

Then we have also for $p=3$ the subadditivity condition.
\begin{theorem} Suppose that $\lambda>\lambda_{0,q}^*$, then for each $0<r_1<r_2$, we have that
	\begin{equation*}
	I_{r_2}<I_{r_1}+I_{r_2-r_1}.
	\end{equation*}
\end{theorem}
\begin{proof} From Theorem \ref{Nehari1} we know that $\lambda_q^*<\lambda_{0,q}^*$ and $I_1<0$, therefore, the conclusion is a consequence of Proposition \ref{mmm2}.
\end{proof}

\section{Constrained Minimization for $p\in(2,10/3)\setminus\{3\}$}\label{sec:Rp}

Now we turn our attention to the existence of minimizers: it is convenient to consider two cases according to the values of $p$:
\begin{itemize}
	\item $p\in(2,3)$,
	\item $p\in(3,10/3)$,
\end{itemize}
although the first case is almost done.

\subsection{The case $p\in (2,3)$ and Proof of Theorem \ref{thmE}}\label{subsec:ppequeno} 
% Casos unificados
%The system has a unique solution $(\widetilde{r}(u),\widetilde{t}(u))$ which satisfies
%\begin{equation}\label{eq:unificar}
%\widetilde{r}(u)=2^{\frac{8-3p}{4(3-p)}}\left(\frac{2(p-2)}{p}\right)^{\frac{1}{2(3-p)}} R_p(u);
%\end{equation}
%\begin{equation}\label{eq:unificat}
%\widetilde{t}(u)^{3p-8}=\left(\frac{p}{2(p-2)}\widetilde r(u)^{\frac{4-p}{2}}\frac{q}{\lambda}\frac{\displaystyle\int \phi_uu^2}{\displaystyle\int |u|^p}\right)^{2}.
%\end{equation}

\medskip

%As a consequence of Theorem \ref{th:boundedbelow} we have
%\begin{proof}[Proof of Theorem \ref{thmE}]  
%\subsection*{Proof of Theorem \ref{thmE}} 
The proof follows immediately from Theorem \ref{stricsub}.

\subsection{The case $p\in (3,10/3)$ and Proof of Theorem \ref{thmp>3}}\label{subsec:pgrande}
%In this Section we prove Theorem \ref{thmp>3}. Unless otherwise stated, we assume $p\in(3,10/3)$
%in all the Section.

By the definitions (see \eqref{eq:rstar}):
$$\forall r> r_0^* : I_r=\inf_{S_{r}}E<0 \quad \text{ and }\quad I_{r_0^*}=\inf_{S_{r_{0}}^{*}}E=0.$$
In both cases the existence of minimizers is already known (see \cite{bellasici1,cattoetal,jeanluo} and also our Theorem \ref{stricsub}). However as we will see $0=\inf_{S_{r}}E<I_r$ if $r\in(r^*,r_0^*)$. 
Let us start with the following
\begin{theorem}\label{theoremp<3n0} %Fix $p\in (3,10/3)$. 
If $(r^*,+\infty)\ni r\mapsto I_r$ is decreasing, then for each $r\in (r^*,r_0^*)$ there exists $u\in \mathcal{N}_r^+\cup \mathcal{N}_r^0$ such that $I_r=E(u)$. 
\end{theorem}
\begin{proof} In fact, let $\{u_n\}\subset  \mathcal{N}_r^+\cup \mathcal{N}_r^0$ be a minimizing sequence to $I_r$. It follows from Proposition \ref{neharibounded} that there exist positive constants $c,C$ such that
	\begin{equation*}
	c\le \|u_n\|\le C,\quad  \forall n\in \mathbb N.
	\end{equation*}
%	Since $u_n\in \mathcal{N}_r^+\cup \mathcal{N}_r^0$ 
	and we conclude that $u_n\nrightarrow 0$ in $L^p(\mathbb{R}^3)$. So $\{u_n\}$ does not vanish and then, up to translations, there exists a subsequence, still denoted by $\{u_n\}$, that converges weakly in $H^1(\mathbb{R}^3)$, strongly in $L^2_{loc}(\mathbb{R}^3)$ and almost everywhere in $\mathbb{R}^3$, to some non-zero function $u\in H^1(\mathbb{R}^3)$.
	
	From  \cite[Lemma 2.2 ]{zhaozhao}, we conclude that
	\begin{equation}\label{min0}
	I_r=\lim_{n\to \infty}E(u_n)=E(u)+\lim_{n\to \infty}E(u_n-u).
	\end{equation}
	Let as usual 
	\begin{equation*}
	Q(u)=	\int |\nabla u|^2+\frac{q}{4}\int \phi_uu^2-\frac{3(p-2)}{2p}\lambda\int |u|^p=0,
	\end{equation*}
	and note that 
	\begin{equation}\label{min1}
	0=\lim_{n\to \infty}Q(u_n)=Q(u)+\lim_{n\to \infty}Q(u_n-u),
	\end{equation}
	and
	\begin{equation}\label{min2}
	\|u\|_2^2=r-\lim_{n\to \infty}\|u_n-u\|_2^2.
	\end{equation}
	We claim that $Q(u)\le0$. On the contrary we would have from \eqref{min1} that $Q(u_n-u)<0$ for sufficiently large $n$. From Proposition  \ref{fiberingmaps}, there exists $t_n>0$ such that 
	$(u_n-u)^{t_n}\in \mathcal{N}_{\|u_n-u\|_2^2}^+$ for large $n$. Once $E(u_n-u)<I_r$ from \eqref{min0} and $\|u_n-u\|_2^2< r$ from \eqref{min2} for sufficiently large $n$, we conclude that
	\begin{equation*}
	I_{\|u_n-u\|_2^2}< E((u_n-u)^{t_n})< E(u_n-u)< I_r,
	\end{equation*}
	which contradicts the hypothesis that $(r^*,\infty)\ni r\mapsto I_r$ is decreasing and therefore $Q(u)\le 0$.
	
	From Proposition \ref{fiberingmaps} there exists $t>0$ such that $u^t\in \mathcal{N}_{\|u\|_2^2}^+\cup \mathcal{N}_{\|u\|_2^2}^0$. Thus, since $E(u)\le I_r$ from \eqref{min0} and $\|u\|_2^2\le r$ from \eqref{min2}, we conclude that
	\begin{equation*}
	I_{\|u\|_2^2}\le E(u^t)\le E(u)\le I_r.
	\end{equation*}
	Therefore, from the hypothesis $(r^*,\infty)\ni r\mapsto I_r$ is decreasing, we conclude that 
	$r=\|u\|_2^2$, $u\in \mathcal{N}_{\|u\|_2^2}^+\cup \mathcal{N}_{\|u\|_2^2}^0$ and  $E(u)=I_r$.
		\end{proof}
In  order to make use of Theorem \ref{theoremp<3n0}, we need to show that $(r^*,+\infty)\ni r\mapsto I_r$ is decreasing. Unfortunately we are able to do so only for some values of $p\in(3,10/3)$, although we conjecture it is true for all $p$ in the range. We note here that in fact, when $I_r<0$ this is a standard result in the literature.
However when $I_r>0$, which is the case for $r\in [r^*, r_{0}^*]$ (see Theorem \ref{Nehari}), the inequalities goes in the opposite direction and thus the proof seems not to be direct.

Our strategy to prove that $I_r$ is decreasing in $r$ will be the following: we will construct paths that crosses the Nehari manifolds when $r$ varies and shows that the energy restricted to theses paths is decreasing. To this end we need to calculate some derivatives.

\begin{lemma}\label{I1} %Suppose that  $p\in(3,{10}/{3})$. 
If $\varphi_{r,u}'(t)=0$ and $\varphi_{r,u}''(t)>0$, then 
	\begin{equation*}
t\int |\nabla u|^2+\frac{r}{2}q\int \phi_uu^2-\frac{3(p-2)}{4}t^{\frac{3p}{2}-4}r^{\frac{p}{2}-1}\lambda\int |u|^p<0.
	\end{equation*}
\end{lemma}
\begin{proof} For simplicity denote 
$$A=\int |\nabla u|^2, \quad B=\int \phi_uu^2 \ \text{ and } \ C=\lambda\int |u|^p.$$
 Since $u\in \mathcal{N}_r^+$ we have that 
	\begin{equation}\label{IMP1}
	\left\{
	\begin{aligned}
	rtA+r^2\frac{q}{4}B-\frac{3(p-2)}{2p}r^{\frac{p}{2}}t^{\frac{3p}{2}-4}\lambda C=0, \\
	rA-\frac{3(p-2)(3p-8)}{4p}r^{\frac{p}{2}}t^{\frac{3p}{2}-5}\lambda C>0. 
	\end{aligned}
	\right.
	\end{equation}
	From the equality in \eqref{IMP1} we conclude that 
	\begin{equation*}
		tA+\frac{r}{2}qB-\frac{3(p-2)}{4}t^{\frac{3p}{2}-4}r^{\frac{p}{2}-1}\lambda C=-tA+\frac{3(p-2)(4-p)}{4p}t^{\frac{3p}{2}-4}r^{\frac{p}{2}-1}\lambda C.
	\end{equation*}
	From the inequality of \eqref{IMP1} we obtain 
	\begin{equation*}
		tA+\frac{r}{2}qB-\frac{3(p-2)}{4}t^{\frac{3p}{2}-4}r^{\frac{p}{2}-1}\lambda C<\frac{3(p-2)(3-p)}{p}t^{\frac{3p}{2}-4}r^{\frac{p}{2}-1}\lambda C<0
	\end{equation*}
	which is the conclusion.
\end{proof}

\begin{corollary}\label{t+diff} %Suppose that $p\in (3,{10}/{3})$, 
Let $I\subset \mathbb{R}$ be an open interval and fix $u\in S_1$. If $t_r^+(u)$ is defined for all $r\in I$, then the function $I\ni r\mapsto t_r^+(u)$ is $C^1$.
\end{corollary}
\begin{proof} Indeed, define $F(r,t)=\varphi'_{r,u}(t)$ for $r\in I$ and $t>0$. From Lemma \ref{I1} it follows that $F(r,t_r^+(u))=0$ and $\frac{\partial F}{\partial r}(r,t_r^+(u))<0$. The proof is then a consequence of the Implicit Function Theorem.
\end{proof}

Consider the equation $-27x^2+146x-192=0$. It has two real roots and the biggest one is given by 
\begin{equation*}
p_0=\frac{73+\sqrt{145}}{27}\in(3,10/3).
\end{equation*} 
\begin{lemma}\label{I2} 
 Assume that  $\varphi_{r,u}'(t)=0$ and $\varphi_{r,u}''(t)>0$. There holds:
 \begin{itemize}
 \item[i)] if $p\in(p_0,{10}/{3})$, then 
  there exists a constant $c_{p}''<0$ such that
	\begin{equation*}
	t^2\int |\nabla u|^2+rtq\int \phi_uu^2-\lambda r^{\frac{p}{2}-1}t^{\frac{3p}{2}-3}\int |u|^p<\frac{c_{p}''}{r^{2}};
	\end{equation*}
 \item[ii)] if $p=p_{0}$ then
	\begin{equation*}
	t^2\int |\nabla u|^2+rtq\int \phi_uu^2-\lambda r^{\frac{p_{0}}{2}-1}t^{\frac{3p_{0}}{2}-3}\int |u|^{p_{0}}<0.
	\end{equation*}	
 \end{itemize}
 
\end{lemma}
\begin{proof} For simplicity denote 
$$A=\int |\nabla u|^2, \quad B=\int \phi_uu^2 \ \text{ and }  \ C=\lambda\int |u|^p.$$
 Since $u\in \mathcal{N}_r^+$ (see Lemma \ref{fibe} and \eqref{eq:decomposicaoN}) we have that 
	\begin{equation}\label{JMP1}
	\left\{
	\begin{aligned}
	rtA+r^2\frac{q}{4}B-\frac{3(p-2)}{2p}r^{\frac{p}{2}}t^{\frac{3p}{2}-4}\lambda C=0, \\
	rA-\frac{3(p-2)(3p-8)}{4p}r^{\frac{p}{2}}t^{\frac{3p}{2}-5}\lambda C>0. 
	\end{aligned}
	\right.
	\end{equation}
	From the equality in \eqref{JMP1} we conclude that
	\begin{equation*}
t^2A+rtqB-\lambda r^{\frac{p}{2}-1}t^{\frac{3p}{2}-3}C=-3t^2A+\frac{5p-12}{p}t^{\frac{3p}{2}-3}r^{\frac{p}{2}-1}\lambda C
	\end{equation*}
	and then, from the inequality in \eqref{JMP1}, we obtain 
	\begin{eqnarray}\label{eq:p0}
t^2A+rtqB-\lambda r^{\frac{p}{2}-1}t^{\frac{3p}{2}-3}C&<&\left(\frac{-9(p-2)(3p-8)}{4p}+\frac{5p-12}{p}\right)\lambda  r^{\frac{p}{2}-1}t^{\frac{3p}{2}-3}C \\
&=&\frac{-27p^2+146p-192}{4p}\frac{\lambda}{r}\int |r^{{1}/{2}}u^t|^p. \nonumber
	\end{eqnarray}
Since by assumptions $u\in \mathcal N_{r}^{+}$ (see Remark \ref{rem:N+0-}), from Proposition \ref{neharibounded} there exists a  constant $c_{p}'>0$ such that 
\begin{equation*}
\int |r^{{1}/{2}}u^t|^p\ge \frac{c'_{p}}{\lambda r},
\end{equation*}
therefore, from the definition of $p_0$, coming back to \eqref{eq:p0}, it follows that
\begin{equation*}
t^2\int |\nabla u|^2+rtq\int \phi_uu^2-\lambda r^{\frac{p}{2}-1}t^{\frac{3p}{2}-3}\int |u|^p
<\frac{-27p^2+146p-192}{4p}\frac{c_{p}'}{r^{2}}:=\frac{c_{p}''}{r^{2}},
\end{equation*}
from which the conclusion easily follows.
\end{proof}

Observe that by \eqref{eq:ccc}, $c_{p}''$ has an  explicit expression.
\begin{proposition}\label{Idecre}  Suppose that $p\in (p_0,{10}/{3})$, then the function $(r^*,\infty)\ni r\mapsto I_r$ is decreasing.
\end{proposition}
\begin{proof} Define $f(r)=E(r^{{1}/{2}}u^{t^{+}_{r}(u)})$ and set for brevity $t(r) = t^{+}_{r}(u)$.
%Then 
%$$f(r) =\varphi_{r,u}(t(r)) = \frac{t(r)^2 r}{2}\int |\nabla u|^2+\frac{r^{2}t(r)}{4}q\int \phi_uu^2-\frac{\lambda}{2}r^{\frac{p}%{2}}t(r)^{\frac{3p}{2}-3}\int |u|^p.$$
Observe from Proposition \ref{diffet+} that $f$ is differentiable and 
	\begin{equation*}
	f'(r)=\frac{t(r)^2}{2}\int |\nabla u|^2+\frac{r^{}t(r)}{2}q\int \phi_uu^2-\frac{\lambda}{2}r^{\frac{p}{2}-1}t(r)^{\frac{3p}{2}-3}\int |u|^p.
	\end{equation*}
	From Lemma \ref{I2} we conclude that $f'(r)\le 2c_{p}''/r^{2}$. Fix $r^*<r_1<r_2$ 
	%and denote $t_i=t_{r_i}^+(u)$ for $i=1,2$
	and $u\in\mathcal{M}_{r_1}$. If $r\in [r_1,r_2]$ then $f'(r)\le 2c_{p}''/r^{2}_{1}$. 
	%where $c$ is some negative constant. 
	Therefore there exists $\theta\in (r_1,r_2)$ such that 
	\begin{equation*}
	f(r_2)-f(r_1)=f'(\theta)(r_2-r_1)\le \frac{2c_{p}''}{r^{2}_{1}} (r_2-r_1),
	\end{equation*}
	and hence
	\begin{equation*}
	E\left(r_2^{{1}/{2}}u^{t(r_{2})}\right)\le E\left(r_1^{{1}/{2}}u^{t(r_{1})}\right)+\frac{2c_{p}''}{r^{2}_{1}}(r_2-r_1),
	\end{equation*}
	which implies that
%	\begin{equation*}
	$I_{r_2}<I_{r_1}$
%	\end{equation*}
and the proof is finished.
\end{proof}
As a consequence of Theorem \ref{theoremp<3n0} and Proposition \ref{Idecre} we have:
\begin{theorem}\label{theoremp<3n0p0} Fix $p\in (p_0,10/3)$, then for each $r\in (r^*,r_0^*)$ there exists $u\in \mathcal{N}_r^+\cup \mathcal{N}_r^0$ such that $I_r=E(u)$. 
\end{theorem}
Now we will show that for $r$ near $r_0^*$ the minimizer found in Theorem \ref{theoremp<3n0p0} belongs to $\mathcal{N}_r^+$. To this end we need to compare the energy of $E$ restricted to $\mathcal{N}_r^0$ with $I_r$.
\begin{lemma}\label{n0bound} For each $r>r^*$, there exists a positive constant $c$ such that
	\begin{equation*}
	E(u)\ge \frac{c}{r}, \quad \forall u\in \mathcal{N}_r^0.
	\end{equation*}
\end{lemma} 
\begin{proof} Indeed by using the pair of equations that characterize $u\in \mathcal{N}_r^0$, that is
		\begin{equation*}
	\left\{
	\begin{aligned}
	rtA+r^2\frac{q}{4}B-\frac{3(p-2)}{2p}r^{\frac{p}{2}}t^{\frac{3p}{2}-4}\lambda C=0, \\
	rA-\frac{3(p-2)(3p-8)}{4p}r^{\frac{p}{2}}t^{\frac{3p}{2}-5}\lambda C=0, 
	\end{aligned}
	\right.
	\end{equation*}
	we can deduce that
	\begin{equation*}
	E(u)=\frac{10-3p}{6(p-2)}\int |\nabla u|^2, \quad u\in \mathcal{N}_r^0,
	\end{equation*}
	therefore from Proposition \ref{neharibounded} the proof is complete.
\end{proof}
\begin{lemma} \label{comparison}There holds 
	\begin{equation*}
	\lim_{r\uparrow r_0^*}I_r=0.
	\end{equation*}
\end{lemma}
\begin{proof} As we observed before  $E_{r_0^*}=I_{r_0^*}$ and there exists $w\in \mathcal{N}_{r_0^*}^+$ with $E(w)=I_{r_0^*}$. Let $u\in S_1$ be such that $w=r_0^*u^{t_{r_0^{*}}^+(u)}$. Since $E(w)=0$ we conclude from the definition of $r_0^*$ and Theorem \ref{nehari83103} that $r_0(u)=r_0^*$. Moreover $r^*=r(u)$. It follows that $t_r^+(u)$ is well defined for each $r\in (r^*,r_0^*)$. Since $I_r\ge 0$ for all $r\in(r^*,r_0^*)$ we obtain from Corollary \ref{t+diff} that
	\begin{equation*}
	0=\lim_{r\uparrow r_0^*}E(ru^{t_{r}^+(u)})\ge \lim_{r\uparrow r_0^*}I_r \ge0
	\end{equation*}
	and the proof is concluded.
	\end{proof}
\begin{theorem}\label{th:ultimo}
For each $p\in (p_0,10/3)$ there exists $\varepsilon>0$ such that for each $r\in (r_0^*-\varepsilon,r_0^*)$,
$I_{r}$ is achieved.
More specifically,  there exists $u\in \mathcal{N}_r^+$ satisfying $I_r=E(u)$. 
\end{theorem}
\begin{proof} From Theorem \ref{theoremp<3n0p0} it remains to prove that $u\in \mathcal{N}_r^+$. Let 
$c/r$ be the constant given by Lemma \ref{n0bound}. Given $0<d<c/r$, from Lemma \ref{comparison} there exists $\varepsilon>0$ such that $I_r<d$ for all $r\in (r_0^*-\varepsilon,r_0^*)$. In particular since 
$I_r<c/r$ it follows that $u\notin\mathcal{N}_r^0$ for all $r\in (r_0^*-\varepsilon,r_0^*)$ and consequently $u\in \mathcal{N}_r^+$.
\end{proof}

%\subsubsection{Proof of Theorem \ref{thmp>3}} 
We can finish now the proof of Theorem \ref{thmp>3}.

\noindent In fact $i)$ follows by  Proposition \ref{Idecre} and Lemma \ref{comparison};
$ii)$ follows by Theorem \ref{theoremp<3n0p0} and 
$iii)$ follows by Theorem \ref{th:ultimo}.

\section{The case $p\in [10/3,6)$}\label{sec:p10/36}
This case was treated in \cite{JBL}, where existence of global minimizers over the Nehari manifold $\mathcal{N}_r^-$ was proved for small $r$. Their proof relies on the fact that for small $r$ the function
\begin{equation*}
J_r=\inf\{E(u):u\in \mathcal{N}_r^-\},
\end{equation*}
is decreasing. 

Fix $u\in S_1$ and define
\begin{enumerate}
	\item[i)] if $p=10/3$ and  $\displaystyle\frac{r}{2}\int |\nabla u|^2-\frac{r^{{p}/{2}}}{p}\lambda\int |u|^p<0$, define $f(r):=\varphi_{r,u}(t_r^-(u))$ for all $r\in(r(u),\infty)$, where $r(u)$ is the unique $r>0$ for which 
	$\displaystyle\frac{r}{2}\int |\nabla u|^2-\frac{r^{{p}/{2}}}{p}\lambda\int |u|^p=0$;
	\medskip
	\item[ii)]   if $p\in (10/3,3)$, define $f(r):=\varphi_{r,u}(t_r^-(u))$ for all $r\in(0,\infty)$. 
\end{enumerate}
Now observe from Lemma \ref{diffet+} that $f$ is $C^1$ and 
\begin{eqnarray*}
f'(r)&=&\varphi_{r,u}'(t_r^-(u))+\frac{1}{2}\left(t_r^-(u)^2\int|\nabla u|^2+rt_r^-(u)q\int\phi_uu^2-\lambda r^{\frac{p}{2}-1}t_r^-(u)^{\frac{3p}{2}-3}\int |u|^p\right).
%&=& \frac{1}{2}\left(t_r^-(u)^2\int|\nabla u|^2+rt_r^-(u)q\int\phi_uu^2-\lambda r^{\frac{p}{2}-1}t_r^-(u)^{\frac{3p}{2}-3}\int |u|^p\right).
\end{eqnarray*}
	For simplicity denote $t_r=t_r^-(u)$. It follows that $f'(r)=0$ if, and only if
\begin{equation}\label{sys3}
\left\{
\begin{aligned}
rt_r\int |\nabla u|^2+\frac{r^2}{4}q\int \phi_uu^2-\frac{3(p-2)}{2p}r^{{p}/{2}}t_r^{\frac{3p}{2}-4}\lambda\int |u|^p=0, \\
t_r^2\int|\nabla u|^2+rt_rq\int\phi_uu^2-\lambda r^{\frac{p}{2}-1}t_r^{\frac{3p}{2}-3}\int |u|^p=0. 
\end{aligned}
\right.
\end{equation}
From Proposition \ref{systemsolve}, system \eqref{sys3} has a unique solution $(\overline{r}(u),\overline{t}(u))$, where
\begin{equation*}
\overline{r}(u)=\left(\frac{2(6-p)}{5p-12}\right)^{\frac{3p-10}{4(p-3)}}\left(\frac{3p}{5p-12}\right)^{\frac{1}{2(p-3)}}R_p(u).
\end{equation*}
Note that $\overline{r}(u)=r(u)$ when $p=10/3$.

\begin{theorem}\label{decreasinirp>10/3} Suppose that $u\in S_1$.
	\begin{enumerate}
		\item[i)] If $p=10/3$, then the function $(r(u),\infty)\ni r\mapsto f(r)$ is increasing. \medskip
		\item[ii)]  If $p\in(10/3,6)$, then the function $(0,\infty)\ni r\mapsto f(r)$ is decreasing for all $r\in (0,\overline{r}(u))$ and increasing for $r\in (\overline{r}(u),\infty)$.
	\end{enumerate}
	
\end{theorem}
\begin{proof} Let $t_r=t_r^-(u)$. It is enough to show that the left hand side of the second equation of system 
	\eqref{sys3} is negative. To this end, by multiplying the first equation by $-4$ and substituting we obtain that 
	\begin{equation*}
	t_r^2\int|\nabla u|^2+rt_rq\int\phi_uu^2-\lambda r^{\frac{p}{2}-1}t_r^{\frac{3p}{2}-3}\int |u|^p=t_r^2h(r),
	\end{equation*}
	where 
	$$h(r)=-3\int |\nabla u|^2+\frac{5p-12}{p}\lambda r^{\frac{p}{2}-1}t_r^{\frac{3p-10}{2}}\int |u|^p.$$
	
	$i)$ This item is direct since for $p=10/3$, the system is linear in $t$.
	\medskip
	
	$ii)$ Since $t_r$ is continuous (see Lemma \ref{diffet+}), it is sufficiently to show that there exist some $0<r_1<\overline{r}(u)<r_2$ such that $h(r_1)<0$ and $h(r_2)>0$. 
	
	We start with the existence of $r_1$. 
	We claim that
	\begin{equation}\label{claim}
	\lim_{r\to 0}h(r)<0.
	\end{equation}
	Indeed, if $t_r$ is bounded from above as $r\to 0$, then \eqref{claim} is obvious, therefore let us assume that $t_r\to \infty$ as $r\to 0$. Note from the first equation of \eqref{sys3} that
	\begin{equation*}
	\int |\nabla u|^2-\frac{3(p-2)}{2p}r^{{p}/{2}-1}t_r^{\frac{3p-10}{2}}\lambda\int |u|^p=o_{r}(1),
	\end{equation*}
	and hence
	\begin{equation*}
	r^{{p}/{2}-1}t_r^{\frac{3p-10}{2}}\lambda\int |u|^p=\frac{2p}{3(p-2)}\int |\nabla u|^2+o_{r}(1).
	\end{equation*}
	Then
	\begin{eqnarray*}
	h(r)&=&-3\int |\nabla u|^2+\frac{5p-12}{p}\lambda r^{\frac{p}{2}-1}t_r^{\frac{3p-10}{2}}\int |u|^p,\\
     	&=&-3\int |\nabla u|^2+\frac{5p-12}{p}\frac{2p}{3(p-2)}\int |\nabla u|^2+o_{r}(1), \\
     	&=& \frac{p-6}{3p-6}\int |\nabla u|^2+o_{r}(1),\quad \mbox{as}\ r\to 0.
	\end{eqnarray*}
Therefore the claim is proved. 

Now we prove the existence of $r_2$. We claim that
\begin{equation}\label{claim1}
\lim_{r\to+ \infty}h(r)>0.
\end{equation}
	Indeed, if $t_r$ is bounded away from $0$ as $r\to 0$, then \eqref{claim1} is obvious, therefore let us assume that $t_r\to 0$ as $r\to \infty$. Note from the first equation of \eqref{sys3} that
\begin{equation*}
\int |\nabla u|^2-\frac{3(p-2)}{2p}r^{{p}/{2}-1}t_r^{\frac{3p-10}{2}}\lambda\int |u|^p=-\frac{r}{4t_r}q\int \phi_uu^2.
\end{equation*}
Since ${r}/{4t_r}\to+ \infty$ as $r\to +\infty$, we conclude that 
$r^{{p}/{2}-1}t_r^{\frac{3p-10}{2}}\lambda\displaystyle\int |u|^p\to+ \infty$ as $r\to +\infty$ and the proof is complete. 
\end{proof}
Let $p\in (10/3,6)$ and for $r,c>0$ define
\begin{equation*}
\mathcal{M}_r=\left\{\frac{u}{\|u\|_2}:u\in \mathcal{N}_r^-\ \mbox{and}\ \int |u|^p\ge c\ \mbox{and}\ \int |\nabla u|^2\le d\right\}.
\end{equation*}
\begin{lemma}\label{decreasingcomcp>10/3} Suppose that $p\in(10/3,6)$ and $r\in [a,b]$ where $0<a<b<\inf_{u\in S_1}\overline{r}(u)$, then there exists a negative constant $c$ such that $f'(r)<c$ for all $u\in  \mathcal{M}_r$ and all $r\in[a,b]$.
\end{lemma}
\begin{proof} In order to prove the lemma, it is sufficiently to prove that the left hand side of the second equation of system \eqref{sys3} is bounded from above by $c$ for all $u\in  \mathcal{M}_r$ and $r\in[a,b]$. From Theorem \ref{decreasinirp>10/3}, we have that $f'(r)<0$ for all $u\in S_1$. Now note that $f(r)=\varphi_{r,u}(t_r^+(u))=\varphi_{r,su}(t_r^+(su))$ for all $s>0$ and therefore, by choosing $s=1/\|\nabla u\|_2$, we can assume that $\|\nabla u\|_2=1$ for all $u\in  \mathcal{M}_r$. Suppose on the contrary that there exists a sequence $\{u_n \}\subset \mathcal{M}_r$ satisfying $\|\nabla u_n\|_2=1$ and corresponding sequences $\{t_n\}\subset(0,+\infty)$, $\{r_n\}\subset [a,b]$ such that 
	\begin{equation*}
	\left\{
	\begin{aligned}
	rt_n\int |\nabla u_n|^2+\frac{r^2}{4}q\int \phi_{u_n}u_n^2-\frac{3(p-2)}{2p}r^{{p}/{2}}t_n^{\frac{3p}{2}-4}\lambda\int |u_n|^p=0, \\
	t_n^2\int|\nabla u_n|^2+rt_nq\int\phi_{u_n}u_n^2-\lambda r^{\frac{p}{2}-1}t_n^{\frac{3p}{2}-3}\int |u_n|^p=o_n(1), 
	\end{aligned}
	\right.
	\end{equation*}
	From Proposition \ref{neharibounded} and the condition $\int|\nabla u_n^{t_n}|^2\le d$ we conclude that $\{t_n\}$ is bounded away from zero and infinity, therefore $\int |u_n|^p$ is bounded away from zero and arguing as in the proof of Lemma \ref{decreasingcomc} we conclude that $r_n=\overline{r}(u_n)+o_{n}(1)\ge \inf_{u\in S_1}\overline{r}(u)+o_{n}(1)>b+\varepsilon+o_{n}(1)$ for some $\varepsilon$, which is a contradiction. The proof is complete.
\end{proof}
\begin{lemma}\label{Jrpro} Suppose that $p\in (10/3,6)$, then $J_r>0$ and every minimizing sequence is bounded and non-vanishing. 
\end{lemma}
\begin{proof} Note that
	\begin{equation}\label{ENEH}
	E(u)=\frac{3p-10}{6(p-2)}\int |\nabla u|^2+\frac{3p-8}{12(p-2)}\int \phi_uu^2, \quad \forall u\in \mathcal{N}_r^-.
	\end{equation}
	Therefore from Proposition \ref{neharibounded} we deduce that $J_r>0$. If $\{u_n\}$ is a minimizing sequence, 
	then from \eqref{ENEH} we conclude that $\{\|\nabla u_n\|_2\}$ is  bounded and hence $\{u_n\}$ is bounded in 
	$H^1(\mathbb{R}^3)$. Moreover this sequence can not be vanishing, since on the contrary, we would obtain from the equation 
	\begin{equation*}
		\int |\nabla u_n|^2+\frac{q}{4}\int \phi_{u_n}u_n^2-\frac{3(p-2)}{2p}\lambda\int |u_n|^p=0,
	\end{equation*}
	that $\displaystyle	\int |\nabla u_n|^2\to 0$ which contradicts with  $J_r>0$.
\end{proof}
\begin{theorem}\label{energydecreasingp>103} The function $(0,\infty)\ni r\mapsto J_r$ is decreasing over the interval $(0,\inf_{u\in S_1}\overline{r}(u))$.
\end{theorem}
\begin{proof}  Fix $0<r_1<r_2< \inf_{u\in S_1},\overline{r}(u)<r^*$ and let $\{u_n\} \subset \mathcal{N}_{r_1}^+$ be a minimizing sequence to $I_{r_1}$. 
	From the mean value theorem we have that 
	\begin{equation*}
	J_{r_2}\le \varphi_{r_2,u_n}(t_{r_2}^-(u_n))= \varphi_{r_1,u_n}(t_{r_1}^-(u_n))+f'(\theta_n)(r_2-r_1), \quad \forall n\in \mathbb N
	\end{equation*}
		where $\theta_n\in (r_1,r_2)$. Note from Lemma \ref{Jrpro} that $\{u_n\}\subset \mathcal{M}_{r_1}$ and therefore $\{u_n\}\subset \mathcal{M}_{r}$ for all $r\in [r_1,r_2]$. From Lemma \ref{decreasingcomcp>10/3} we conclude that $f'(\theta_n)(r_2-r_1)<c(r_2-r_1)$ where $c<0$. As a consequence
 	\begin{equation*}
	J_{r_2}\le 	J_{r_1}+c(r_2-r_1),
	\end{equation*}
	and the proof is complete.
\end{proof}
From \cite{JBL} we conclude
\begin{theorem} For each $r\in (0,\inf_{u\in S_1}\overline{r}(u))$, there exists $u\in \mathcal{N}_r^-$ such that $J_r=E(u)$.
\end{theorem}
%\begin{proof} %	In fact, let $\{u_n\}\subset  \mathcal{N}_r^-$ be a minimizing sequence to $J_r$. From Proposition \ref{Jrpro} we can assume that $u_n$ is bounded and non-vanishing and therefore, there exists a subsequence, still denoted by $\{u_n\}$, that converges weakly in $H^1(\mathbb{R}^3)$, strongly in $L^2_{loc}(\mathbb{R}^3)$ and almost everywhere in $\mathbb{R}^3$, to some non-zero function $u\in H^1(\mathbb{R}^3)$.
%	
%	From  \cite[Lemma 2.2]{zhaozhao}, we conclude that
%	\begin{equation}\label{min00}
%	J_r=\lim_{n\to \infty}E(u_n)=E(u)+\lim_{n\to \infty}E(u_n-u).
%	\end{equation}
%	We also have
%	\begin{equation}\label{min22}
%	\|u\|_2^2=r-\lim_{n\to \infty}\|u_n-u\|_2^2.
%	\end{equation}
%	Denote $t=t_{\|u\|_2^2}^-(u)$ and observe that
%	\begin{eqnarray*}
%	J_{\|u\|_2^2}&\le& E(u^t), \\
%	&=&\lim_{n\to \infty}E(u_n^t)-\lim_{n\to \infty}E((u_n-u)^t), \\
%	&\le & J_r,
%	\end{eqnarray*}
%	and from Theorem \ref{energydecreasingp>103} we conclude that $	J_{\|u\|_2^2}=J_r$, $\|u\|_2^2=r$ and $J_r=E(u)$.
%\end{proof}

\section{Appendix. New inequalities}\label{sec:ineq}
We conclude with some estimates; in particular the second one is new in the literature.
\begin{theorem}\label{estima} There hold:
	\begin{enumerate}
		\item[i)] for each $p\in(2,3)$, there exists a positive function ${\mathfrak f}_p:(0,\infty)^2\to \mathbb{R}$ such that if $r_1<r_2$, then
		\begin{equation*}
		I_{r_2}>\left(\frac{r_2}{r_1}\right)^3I_{r_1}+\mathfrak f_{p}(r_1,r_2)\lambda\left(\frac{r_1}{r_2}\right)^{p-3} \left[\left(\frac{r_1}{r_2}\right)^{2(p-3)}-1\right];
		\end{equation*}
		\item[ii)] for each $p\in (3,10/3)$ and $r^*<r_1<r_2$, then %there exists a  constant $c_p>0$ such that 
		\begin{equation*}
		I_{r_2}<\left(\frac{r_2}{r_1}\right)^{3}I_{r_1}-\frac{c'_p}{r_1}\left(\frac{r_2}{r_1}\right)^{p} \left[\left(\frac{r_2}{r_1}\right)^{2(p-3)}-1\right],
		\end{equation*}
		where $c_{p}'>0$ is the constant given in Proposition \ref{neharibounded}.
	\end{enumerate}
\end{theorem}

\begin{proof} $i)$ Indeed, fix $r_1<r_2$
and ake $u\in \mathcal{M}_{r_2}$ satisfying $E(r_2^{{1}/{2}}u^{t(r_2)})<0$. 
For simplicity we set $t_i=t(r_i)$, $i=1,2$.
From Lemma \ref{containe} we know that $u\in \mathcal{M}_{r_1}$ and $E(r_1^{{1}/{2}}u^{t_1})<0$, which implies that $t_{1}$ is a global minimum for the fiber map $\varphi_{r_1,u}$ and therefore 
	\begin{eqnarray*}%\label{pp1}
	E(r_2^{{1}/{2}}u^{t_2}) &=& \frac{r_2^3}{r_1^3}E(r_1^{{1}/{2}}u^{t_2\frac{r_1}{r_2}})+\frac{\lambda}{p}\left(\frac{r_1}{r_2}\right)^{p-3} \left[\left(\frac{r_1}{r_2}\right)^{2(p-3)}-1\right]\int|r_2^{{1}/{2}}u^{t_2}|^p,  \\
	&> & \frac{r_2^3}{r_1^3}E(r_1^{{1}/{2}}u^{t_1})+\frac{\lambda}{p}\left(\frac{r_1}{r_2}\right)^{p-3} \left[\left(\frac{r_1}{r_2}\right)^{2(p-3)}-1\right]\int|r_2^{{1}/{2}}u^{t_2}|^p. \nonumber
	\end{eqnarray*}
	Since $p\in(2,3)$, it follows that $\left({r_1}/{r_2}\right)^{2(p-3)}-1>0$, therefore if $\{u_n\}\subset \mathcal{M}_{r_2}$ is choosen in such a way that $\{ r_2^{{1}/{2}}u_n^{t_2} \}$ is a minimizing sequence for $I_{r_2}$, since it must be non-vanishing we obtain that
\begin{equation*}
I_{r_2}>\frac{r_2^3}{r_1^3}I_{r_1}+\mathfrak f_{p}(r_1,r_2)\lambda\left(\frac{r_1}{r_2}\right)^{p-3} \left[\left(\frac{r_1}{r_2}\right)^{2(p-3)}-1\right].
\end{equation*} 

\medskip

	$ii)$ Indeed, fix $r^*<r_1<r_2$ and take $u\in \mathcal{M}_{r_1}$. For simplicity, let again $t_i=t(r_i)$ for $i=1,2$ and set
	\begin{equation*}
	Q(u):=\int |\nabla u|^2+\frac{q}{4}\int \phi_uu^2-\frac{3(p-2)}{2p}\lambda\int |u|^p.
	\end{equation*}
	Observe that
	\begin{eqnarray*}
	Q(r_1^{{1}/{2}}u^{t_1})&=&t_1r_1\int |\nabla u|^2+\frac{r_1^2}{4}q\int \phi_uu^2-\frac{3(p-2)}{2p}t_1^{\frac{3p}{2}-4}r_1^{\frac{p}{2}}\lambda\int |u|^p \\
	&=&\frac{r_1^3}{r_2^3}Q(r_2^{{1}/{2}}u^{t_1\frac{r_2}{r_1}})+\frac{3(p-2)}{2p}\lambda\left(\frac{r_2}{r_1}\right)^{p-3} \left[\left(\frac{r_2}{r_1}\right)^{2(p-3)}-1\right]\int|r_1^{{1}/{2}}u^{t_1}|^p.
	\end{eqnarray*}
Since $Q(r_1^{{1}/{2}}u^{t_1})=0$, $r_1<r_2$ and $p>3$, we conclude that
\begin{equation*}
Q(r_2^{{1}/{2}}u^{t_1\frac{r_2}{r_1}})<0,
\end{equation*}
and hence from Proposition \ref{fiberingmaps} item $III-1)$, it follows that $t_r^-(u)<t_1\frac{r_2}{r_1}<t_2$. Therefore
\begin{eqnarray}\label{p1}
E(r_1^{{1}/{2}}u^{t_1}) &=& \frac{r_1^3}{r_2^3}E(r_2^{{1}/{2}}u^{t_1\frac{r_2}{r_1}})+\frac{\lambda}{p}\left(\frac{r_2}{r_1}\right)^{p-3} \left[\left(\frac{r_2}{r_1}\right)^{2(p-3)}-1\right]\int|r_1^{{1}/{2}}u^{t_1}|^p,  \\
&> & \frac{r_1^3}{r_2^3}E(r_2^{{1}/{2}}u^{t_2})+\frac{\lambda}{p}\left(\frac{r_2}{r_1}\right)^{p-3} \left[\left(\frac{r_2}{r_1}\right)^{2(p-3)}-1\right]\int|r_1^{{1}/{2}}u^{t_1}|^p. \nonumber
\end{eqnarray}
Since $	r_1^{{1}/{2}}u^{t_1}\in \mathcal{N}_{r_1}$, it follows from Proposition \ref{neharibounded} that there exists  a constant $c_{p}'>0$ such that
\begin{equation*}
\int|r_1^{{1}/{2}}u^{t_1}|^p\ge \frac{c'_p}{\lambda r_1}, \quad \forall u\in \mathcal{M}_{r_1},
\end{equation*}
and consequently from \eqref{p1} we conclude that
\begin{equation*}
E(r_1^{{1}/{2}}u^{t_1})\ge  \frac{r_1^3}{r_2^3}E(r_2^{{1}/{2}}u^{t_2})+\frac{c'_p}{r_1}\left(\frac{r_2}{r_1}\right)^{p-3} \left[\left(\frac{r_2}{r_1}\right)^{2(p-3)}-1\right], \quad \forall u\in \mathcal{M}_{r_1}.
\end{equation*}
Therefore
\begin{equation*}
I_{r_2}<\left(\frac{r_2}{r_1}\right)^{3}I_{r_1}-\frac{c'_p}{r_1}\left(\frac{r_2}{r_1}\right)^{p} \left[\left(\frac{r_2}{r_1}\right)^{2(p-3)}-1\right]
\end{equation*}
which concludes the proof.
\end{proof}

\medskip

%------------------------------------------------------
%------------------------------------------------------

\end{document}